\documentclass[12pt]{amsart}
\usepackage{graphicx}
\usepackage{amsmath,amsbsy}
\usepackage{amssymb, enumerate,color}
\usepackage{comment}
\usepackage{mathrsfs}      
\usepackage{helvet}         
\usepackage{courier}        
\usepackage{type1cm}        
\usepackage{esint}
\usepackage{stmaryrd}

\allowdisplaybreaks

\includecomment{versiona}

\long\def\NOTE#1{} 

\usepackage{geometry,calc}
\newtheorem{thm}{Theorem}[section]
\newtheorem{cor}[thm]{Corollary}
\newtheorem{lem}[thm]{Lemma}

\theoremstyle{definition}

\newtheorem{pb}[thm]{Problem}

\theoremstyle{remark}
\newtheorem{rem}[thm]{Remark}
\numberwithin{equation}{section}

\renewcommand{\L}{\Phi}
\newcommand{\Ph}{\Psi}

\renewcommand{\u}{h}
\renewcommand{\v}{w}
\renewcommand{\l}{\varphi}
\newcommand{\ph}{\psi}
\newcommand{\m}{\mu}
\newcommand{\ps}{\zeta}

\newcommand{\T}{S}

\newcommand{\sumi}{{ \sum_{i=1}^L}}
\newcommand{\NF}{L}
\newcommand{\NT}{M}

\newcommand{\Sq}{\Sigma}

\newcommand{\matR}{\mathbf{R}}
\newcommand{\matS}{\mathbf{S}}
\newcommand{\matM}{\mathbf{M}}

\newcommand{\Traces}{\mathscr{T}}
\newcommand{\Fractures}{\mathscr{F}}

\newcommand{\Humi}{H^{1/2}_{F_i}}
\newcommand{\Hmumi}{H^{-1/2}_{F_i}}

\newcommand{\summi}{{\sum_{m\in \Ti}}}
\newcommand{\summ}{{\sum_{m=1}^M}}

\newcommand{\calA}{\mathcal{A}}
\newcommand{\calV}{\mathcal{V}}
\newcommand{\calS}{\mathcal{S}}

\newcommand{\calG}{\mathcal{G}}
\newcommand{\stabweight}{\omega}
\newcommand{\calSnew}{\mathcal{S}^h}

\newcommand{\calGnew}{\mathcal{G}^h}

\newcommand{\gsep}{\gamma_0}
\newcommand{\huh}{\widehat \u_\delta}

\newcommand{\normoneh}[1]{\| #1 \|_{1,\h}}

\newcommand{\Basis}{\mathscr{B}}
\newcommand{\tBasis}{\widetilde{\Basis}}
\newcommand{\hBasis}{\widehat{\Basis}}

\newcommand{\F}{F}
\newcommand{\Th}{\mathcal{T}_\delta}
\newcommand{\TSq}{\mathcal{K}_\delta}

\newcommand{\supp}{\mathrm{supp}\,}

\newcommand{\Thi}{{\Th^i}}
\newcommand{\Edges}{\mathcal{E}_\h^i}
\newcommand{\h}{\delta}
\newcommand{\nK}{\nu_T}
\newcommand{\K}{T}
\newcommand{\bK}{\partial \K}
\newcommand{\jump}[1]{\llbracket #1 \rrbracket}
\newcommand{\scjump}[1]{[ #1  ]}
\newcommand{\norm}{\nu}
\newcommand{\Fi}{{\F_i}}
\newcommand{\Sqi}{{\Sq_i}}

\renewcommand{\O}{\Omega}
\newcommand{\Lh}{\L_\h}
\newcommand{\Phh}{\Ph_\h}
\newcommand{\uh}{\u_\h}
\newcommand{\vh}{\v_\h}
\newcommand{\lh}{{\l_\h}}
\newcommand{\phh}{\ph_\h}
\newcommand{\mh}{\m_\h}
\newcommand{\psh}{\ps_\h}
\newcommand{\Wauxi}{W_\h^i}
\newcommand{\Wauxim}{W_\h^{i,m}}
\newcommand{\wh}{\widehat w_\h}
\newcommand{\Taux}{\widehat{\mathcal{T}}}

\newcommand{\hw}{\widehat w}
\newcommand{\tw}{\widetilde w}

\newcommand{\auxb}{\eta}

\newcommand{\Nodesaux}{\widehat{\mathcal{N}}}
\newcommand{\weight}{\theta}
\newcommand{\stab}[2]{[#1,#2]_{i,m,\h}}
\newcommand{\stabi}[2]{{[#1,#2]}_{-1,i,\h}}
\newcommand{\glstab}[2]{[#1,#2]_{-1,\h}}

\newcommand{\A}{\mathfrak A}
\newcommand{\B}{\mathfrak B}
\newcommand{\Ti}{\mathscr{T}_i}
\newcommand{\Fm}{\mathscr{F}_m}

\newcommand{\tM}{\widetilde \matM^m}
\newcommand{\hExt}[1]{\widehat{\mathcal{E}}(#1)}

\newcommand{\ccoerca}{\kappa_1}
\newcommand{\ccoercb}{\kappa_2}
\newcommand{\const}[1]{C_{#1}}

\newcommand{\lhi}{\l_\h^i}
\newcommand{\lhim}{\l_\h^{i,m}}

\begin{document}

\title[]{A Stabilized Three Fields formulation for \\ Discrete Fracture Networks}%

 \author{Stefano Berrone}%
\address{DISMA -- Politecnico di Torino}%
\email{stefano.berrone@polito.it}%

 \author{Silvia Bertoluzza}%
\address{IMATI-CNR}%
\email{silvia.bertoluzza@imati.cnr.it}%

 \author{Stefano Scial\`o}%
\address{DISMA -- Politecnico di Torino}%
\email{stefano.scialo@polito.it}%

\date{\today}
\thanks{This paper was funded by the MIUR Progetti di Ricerca di Rilevante Interesse Nazionale (PRIN) Bando 2022 PNRR, grant P2022BH5CB (funded by the European Union - Next Generation EU). The authors are members of the Gruppo Nazionale Calcolo Scientifico -- Istituto Nazionale di Alta Matematica (GNCS-INdAM)}

\maketitle

\begin{abstract}
	We propose a hybridized domain decomposition formulation of the discrete fracture network model, allowing for independent discretization of the individual fractures. A natural norm stabilization, obtained by penalizing the residual measured in the norm for the space where it naturally ``lives'',  is added to the local problem in the individual fracture so that no compatibility condition of inf-sup type is required  between the Lagrange multiplier and the primal unknown, which can then be discretized independently of each other. Optimal stability and error estimates are proven, which are confirmed by numerical tests.
\end{abstract}

%

\renewcommand{\omega}{\alpha}

\section{Introduction}

Discrete Fracture Networks, or DFNs, are networks of intersecting planar polygons, resembling the fractures in a porous medium. DFN models can be used in flow simulations of the subsoil, and in this case they are composed, typically, by a large number of polygonal fractures with arbitrary intersections, forming complex networks with a strong multi-scale structure. Indeed fracture dimensions can range from centimeters to kilometers and form intersections, that are one dimensional manifolds, wich can also span several lenght scales. Further, such networks are generated from probability distribution functions of hydraulic and geological soil properties \cite{CLMTBDFP,DeDreuzyEtAlii01,Dowd2007},  so that, in practical cases, a large number of simulations is required to obtain reliable statistics of the quantities of interest. 

DFNs are particularly effective when the presence of fractures can sensibly alter relevant flow properties, such as flow directionality and preferential flow paths. In these circumstances homogenization based approaches \cite{Berkowitz2002} might not provide reliable outcomes \cite{Fidelibus2009,Sahimi2011}. On the other hand geometrical complexity and size of the resulting computational domain might be a limiting factor for the use of DFNs. Standard approaches, indeed, rely on mesh conformity at fracture intersections to enforce matching conditions, but the generation of good quality conforming meshes might turn out to be unfeasible for real applications \cite{deDreuzy13,Antonietti2016,FKS}, for the above mentioned complexities. 

Different approaches have been suggested to overcome the difficulties in DFN simulations. In \cite{DF99,BODIN2007}, for example, the DFN is replaced by a set of one-dimensional channels resembling the connections among fractures in the network, whereas in \cite{NOETJCP15} the solution in the interior of the fractures is written in terms of the interface unknowns, thus reducing problem size. Other works suggest the use of graph theory tools to efficiently handle DFN simulations \cite{Karra2018,HOBE2018,HymanMLNature}. 
Other approaches, instead focus on the development of new methods to efficiently compute conforming meshes. A possibility consists in introducing small modifications of the network geometry to ease meshing \cite{Fourno2019}. For stochastically generated networks, instead, hard-to-mesh networks can be replaced by equivalent analogues, as proposed in \cite{Gable2014,Gable2015bTransp}. Numerical methods based on polygonal/polyhedral meshes are also attractive, for the possibility of easily generating and handling meshes of complex domains: examples of DFN simulations are available based on Virtual Elements in \cite{BBBPS,BBSmixed,Fumagalli2018}, or with Mimetic Finite Differences in \cite{Antonietti2016} and Hybrid High-Order Methods in \cite{Chave2018}. 
Domain decomposition approaches are also explored, based on mortar methods \cite{Vohralik2007,Pichot2012,Pichot2014} and allowing a partial non-conformity of the mesh with the interfaces, or on a PDE-constrained optimization method \cite{BPSa,BPSc,BGPS}, allowing completely non conforming meshes at fracture intersections.

In the present paper we consider a hybridized formulation  of the coupled equations along the lines of \cite{Brezzi}. This is obtained by introducing the trace of the solution
on the interface as an independent unknown, and writing both the equation and the coupling conditions in weak form. On each fracture, a Darcy problem is independently solved, subject to a  constraint on the intersection with other fractures, forcing the trace of the solution to coincide with the unknown trace variable. The constraint is imposed by means of Lagrange multipliers, local to each fracture, and coupled to each other by the flux conservation property.  In the domain decomposition framework, this formulation has been shown to hold great potential. By allowing, in principle, for the use of independent discretizations for the subdomains and for the interface, it allows a great freedom in the design of the meshes, which can be constructed independently on each subdomain. Moreover, quasi optimal preconditioners are available  \cite{bertoluzza2000wavelet,bertoluzza2004substructuring}, allowing for the efficient solution of the resulting linear system.

To be able fully exploit the freedom to choose completely independent mesh in the different fractures, that in principle such a formulation  allows, it is convenient to resort to some kind of stabilization. Indeed, by  gaining some control on either the trace unknown or on the Lagrange multiplier, we manage to avoid the need of simultanously satisfying, at the discrete level, both the two compatibility conditions required for the stability of the problem at the continuous level.  Different stabilization approaches can be employed to gain such control.	Bubble stabilization has been proposed in \cite{brezzi1997stabilization,brezzi2001error,buffa2002error}, while a full stabilization, consisting in a number of different terms accounting for different possible sources of instability can be found in \cite{baiocchi1992stabilization}. The typical approach for such kind of problem would consist in the addition of mesh dependent penalization terms, similar to the stabilization terms appearing in Nitsche's method \cite{Nitsche1971uber} or in interior penalty methods (\cite{burman2005unified}), acting  on the discrepancy at the interface between the solution in the individual fractures and the additional trace variable, and/or on the discrepancy between the multiplier and the flux (as in the Barbosa-Hughes method \cite{Barbosa1992circumventing}).  
We point out that a straightforward application of the Barbosa-Hughes stabilization would require the mesh to be somhow conforming with the interface (that is, the interface needs to be a subset of the union of the edges of the mesh), a requirement that greatly limits the freedom in the mesh design.
Following the approach proposed in \cite{bertoluzza2022algebraic}, we resort instead here to a computable dual norm stabilization local to the individual fractures, of which we also present a mesh dependent version, obtained by simple algebraic manipulations. The latter can be considered the equivalent of the Barbosa-Hughes stabilization in the case where the mesh is not conforming to the interface.

The paper is organized as follows. In Sections \ref{sec:2} and \ref{sec:3} we respectively present the hybridized formulation and the proposed stabilized discretization, and state the first stability and convergence result. In Section \ref{sec:StabTerm} we discuss the construction of the natural norm stabilizing term, while in section \ref{sec:stab_meshdependent} we derive the new mesh dependent stabilization, by some algebraic manipulations on the stabilizing term, and we prove the corresponding stability and convergence results. In Section \ref{sec:numtests} we present some numerical experiment, confirming the validity of the theoretical results.

\section{Three fields formulation of the DFN model}\label{sec:2}

%
%

\newcommand{\bO}{\partial\Omega}
\newcommand{\bON}{\bO^N}
\newcommand{\bOD}{\bO^D}

Let $\bar \Omega = \cup_{i=1}^\NF \bar\F_i \subset \mathbb{R}^3$ denote a Discrete Fracture Network (DFN), union of planar  polygonal fractures $F_i$. We let $\partial \Omega$ denote the boundary of $\Omega$, which we assume to be split as $\bO = \bO^N \cup \bO^D$.
The intersections of fractures are called traces and denoted by $\T_m$, $m = 1,\cdots \NT$. On $\O$ we consider the following system of coupled partial differential equations for the hydraulic head $\u$: 
\begin{gather}
	- \nabla \cdot K_i \nabla \u^i = g^i, \qquad \text{ in }\Fi, \label{strong1}\\
	\u^i = \u^j\qquad \text{ on } \F_i \cap F_j,\label{strong1}\\
	\jump{K_i \nabla \u^i \cdot \nu^i} =  \jump{K_j \nabla \u^j \cdot \nu^j} \qquad \text{ on } \F_i \cap F_j,\label{strong3}\\
	\u^i = 0 \ \text{ on } \partial F_i \cap \bOD, \quad \text{ and } \quad K_i \nabla \u^i \cdot \nu^i = 0 \text { on } \partial F_i \cap \bON.\label{strong4}
	\end{gather}
In the above equation  $\jump{K_i \nabla u^i \cdot \nu^i}$ denotes the flux, that is, the jump of the conormal derivative of the hydraulic head $\u_i$ along the direction  parallel to $F_i$ and orthogonal to $ \F_i \cap F_j$. More precisely, $\jump{K_i \nabla u^i \cdot \nu^i} = K_i \nabla u^i_+ \cdot \nu^i_+ + K_i \nabla u^i_- \cdot \nu^i_- $, where we label with $+$ and $-$ the two sides of the trace $F_i \cap F_j$ and we let $\nu^i_+$ and $\nu^i_-$  are the normal directions parallel to $F_i$ and normal to $S_m$, respectively pointing outwards from the $+$ and the $-$ side.   For the sake of simplicity we assume that $K_i$ is symmetric and constant on each fracture.

\


We let $\Fractures = \{ \Fi, \ i=1,\cdots,N\}$ denote the set of all fractures, and $\Traces = \{S_m = F_{i_m} \cup F_{j_m}\}$ the set of all traces. Moreover we let
$\Fm \subset \Fractures$ denote the set of those fractures that share $S_m$ as a trace, and $\Ti$ the set of the traces lying  on the fracture $\F_i$.  We define the skeleton $\Sq = \cup_m S_m$ as the union 
of all  traces, and we let $\Sq_i = \Sq\cap {\F_i}$ denote its restriction to the fracture $F_i$. We let 
\[
	H^1_*(\Fi) = \{ \v \in H^1(\F_i):\ \v = 0\ \text{ on }\partial\Fi \cap \bOD \},\qquad  \qquad H^{-1}_*(\Fi) = (H^1_*(\F_i))',
\]
and we let the global $H^1_*(\O)$ be defined as
\[
H^1_*(\O) = \{
h: \u^i = h|_{\Fi} \in 	H^1_*(\Fi), \ \text{ and } \  \u^i = \u^j \text{ on } \Fi \cap \F_j
\}.
\]
%
%

We define  spaces $\Ph$, $\Ph^i$, $\L$ and $\L^i$ as
\[
\Ph = H_*^1(\O)|_\Sq, \qquad \Ph^i = H_*^1(\F^i)|_\Sqi, \qquad 
\L^i  =  (\Ph^i)',\qquad \L = \prod_i \L^i,\]
respectively endowed with the norms
\[
\| \ph \|_{\Ph} = \inf_{{v \in H_*^1(\O)}\atop
	{\ph = v|_\Sq}} \| v \|_{1,\O}, \qquad \| \ph \|_{\Ph^i} = \inf_{{v \in H^1_*(\F^i)}\atop
	{\ph = v|_{\F^i}}} \| v \|_{1,\F^i}, \qquad \| \l \|_{\L^i} = \sup_{\ph \in \Ph^i} \frac{\int_{\Sqi} \l \ph   }{\| \ph \|_{\Ph^i}},
\]
$\L$ being endowed with the product norm. Remark that the space $\Ph$ is single valued on the skeleton, while the space $\L$ is multi-valued. 

\

In this paper we limit ourselves to present our approach under the simplifying assumptions that the traces are not only all disjoint from each other, but that they are also well separated. More precisely, we assume that there exists a constant $\gsep$ such that for all $\T, \T' \in \Traces$, $\T \not = \T'$,  
\begin{equation}\label{separation}
	\min_{x\in \T, y\in \T'} | x - y | \geq \gsep.
\end{equation}

Under this assumption, the traces $S_m$  can be handled independently one from the other, which allows to characterize $\Ph_i$ and $\L_i$ as
\begin{equation*}
	\Ph_i = \prod_{m\in\Ti} \Humi(S_m), \qquad \L_i = \prod_{m \in \Ti}\Hmumi(S_m),
\end{equation*}
endowed with the corresponding product norm, where $\Humi(S_m) \subseteq H^{1/2}(S_m)$ is defined as 
\[
\Humi(S_m) = H^1_*(F_i)|_{S_m}.
\]

\

In order to write down our discretization method, we start from the following continuous formulation of the DFN problem
\begin{pb}\label{pb:3fcont}
	 Find $\u = (\u^i)_i \in V = \prod_i H^1_*(\Fi)$, $\l = (\l^i)_i \in \L = \prod_i \L^i$, $\ph \in \Ph$ such that 
	\begin{gather*}
		\int_{\Fi} K_i \nabla \u^i \cdot \nabla \v^i - \int_\Sqi \l^i \v^i = \int_{\Fi} g^i \v^i, \qquad \forall \v^i \in H^1_*(\Fi), \ \forall i,\\
		\int_\Sqi \u^i \m^i - \int_\Sqi \ph \m^i = 0, \qquad \forall \m^i \in \L^i, \ \forall i,\\
		\sumi  \int_{\Sqi} \l^i \ps = 0, \qquad \forall \ps \in \Ph.
		\end{gather*}
		\end{pb}
		It is not difficult to ascertain that Problem \ref{pb:3fcont} is equivalent to (\ref{strong1}--\ref{strong4}).

\

\section{Discretization}\label{sec:3}
We now let $\Th^i$ and $\TSq^m$ respectively denote quasi uniform shape regular meshes of $\Fi \in \Fractures$ and $S_m \in \Traces$, with mesh size $\h_i$ and $\h_m$. 
For the sake of computational efficiency,  we would like the meshes $\Th^i$ to be constructed independently  of each other,  and  of the meshes $\TSq^m$. Therefore, we do not assume that the different meshes are matching at the interface.  For the sake of simplicity, we however assume that the mesh size of all the different meshes are comparable, that is that $\h_i \simeq \h_m \simeq \h$\footnote{
In the following we will employ the notation $A \lesssim B$ (resp. $A \gtrsim B$) to indicate that the quantity $A$ is less (resp. greater) or equal than the quantity $B$ times a constant independend of the mesh size, but possibly depending on the shape regularity of the mesh. 
}.

We introduce independent local order one discretization spaces $V_\h^i \subset H^1_*(\Fi)$, defined as 
	\begin{equation}
			\label{defVh}	V^i_\h = \{\uh \in H^1_*(\Fi): \uh \in \mathbb{P}_1(T), \ \forall T \in \Th^i \},
	\end{equation}
and local discretization spaces  $\Ph^m_\h \subset \Ph^m$ for $\ph$ on each trace $S_m \subseteq \Sq$, defined as
\begin{equation}
		\Phh^m = \{ \ph \in C^0(S_m): \ \ph \in \mathbb{P}_1(e),\ \forall e \subseteq \TSq^m, \ \ph = 0\ \text{on}\ \bar S_m \cap \bOD\}.\label{defPh}
\end{equation}
Moreover,  for all traces $S_m \in \Traces$, $\Fi \in \Fm$ , we define local discretization spaces $\Lh^{i,m} \subset \Hmumi(S_m) = (\Humi(S_m))'$ for the different multipliers $\l^i$ as
	\begin{equation}
			\label{eq:defLh}	\Lh^{i,m} = \Phh^m.
	\end{equation}

We then assemble the global discrete spaces as
\[\underline{}
V_\h = \prod_{\Fi \in \Fractures} V_\h^i,\quad \L_\h = \prod_{F_i \in \Fractures} \L_\h^i = \prod_{F_i \in \Fractures} (\prod_{S_m\in \Ti} \L_\h^{i,m}), \quad
\Ph_\h = \prod_{S_m \in \Traces} \Ph_\h^m.
\]

\begin{rem}
	We remark that, by taking $\L_\h^{i,m} = \Ph_\h^m$ we introduce an error of order $\h^{1/2} \log(\h)$ in our discretization scheme, whenever one or both the extrema of $S_m$ lie on the Dirichlet boundary $\bOD$, as $\L_\h^{i,m}$ needs to approximate the flux $\jump{K_i\nabla \u^i}$, which does not generally vanish on $\bOD$. This could be avoided by constructing $\L_\h^{i,m}$ with the approach given by the mortar method. However, the solution of our problem typically presents jumps in the conormal derivative along traces, which also yield an error of order $\h^{1/2}\log(\h)$, so that our choice does not degrade the order of convergence of the overall scheme. It will instead allow, later on, for some simplification that will allow to  define, in Section \ref{sec:stab_meshdependent}, a particularly simple and effective mesh dependent stabilization strategy.
\end{rem}

A straightforward discretization of Problem \ref{pb:3fcont} yields the following discrete problem: find $\uh = (\uh^i) \in V_\h$, $\lh = (\lh^i)_i \in \Lh$, $\phh \in \Ph_\h$ such that  
\begin{gather}
	\int_{\Fi} K_i \nabla \uh^i \cdot \nabla \vh^i - \int_\Sqi \lhi  \vh^i = \int_{\Fi} g^i \vh^i, \qquad \forall \vh^i \in V_\h^i, \ \forall i,  \label{nonstab1}\\
\int_\Sqi \uh^i \mh^i - \int_\Sqi \phh \mh^i = 0, \qquad \forall \mh^i \in \Lh^i, \ \forall i, \label{nonstab2}\\
\sumi  \int_{\Sqi} \lhi  \psh = 0, \qquad \forall \psh \in \Phh.\label{nonstab3}
\end{gather}
It is well known that the well posedness of the above problem relies on the  validity of two inf-sup conditions, which might be quite difficult to satisfy simultaneously. On the other hand, the mesh $\TSq$ is shape regular, our choice of discretization implies, by some standard argument, that
\begin{equation}\label{infsupbase}
	\inf_{\phh \in \Phh} \sup_{\lh \in \Lh} \frac{ \sumi  \int_{\Sqi} \phh \lh }{
		\| \phh \|_{\Ph} \| \lh \|_{\L}
	} \gtrsim 1,
\end{equation}

In general, condition \eqref{infsupbase} is not sufficient for the stability of the discretized problem, and, consequently, we need to add some stabilization. We propose to add independent  stabilization terms  to  the local equations \eqref{nonstab1}--\eqref{nonstab2} in the unknowns $\uh^i$, $\lhi $ on the individual fractures,
 by following the general approach proposed in \cite{bertoluzza2022algebraic}. To this aim we introduce  the linear operators corresponding to the bilinear forms involved in the continuous formulation, namely we set
\begin{gather*}
	A_i: H^1_* (\Fi) \to H^{-1}_* (\Fi), \qquad \langle A_i \u, \v \rangle = \int_\Fi K_i \nabla \u \cdot \nabla \v,\\
	B_{i,m}: H^1_* (\Fi) \to \Hmumi(S_m), \qquad B_{i,m} = \gamma^m_i  \vh,\\
	B_i: H^1_*(\Fi) \to \L'_i, \qquad \langle B_i v, \l^i \rangle = \sum_{m \in \Ti}  \langle \l^ {i,m}, \gamma_i^m v \rangle_{S_m}
\end{gather*}
where, for $S_m \in \Ti$,  $\gamma^m_i : H^1_* (\Fi) \to \Humi(S_m)$ is the trace operator.  We let 
$B_{i,m}^T: \Hmumi (S_m) \to H^{-1}_* (\Fi)$ denote the adjoint of $B_{i,m}$, which coincides with the canonical injection of $\Hmumi(S_m)$ in $H^{-1}_*(\Fi)$. Letting $\calV_i = H^1_*(\Fi) \times \L_i$, and letting \(\calA_i : \calV_i \times \calV_i \to \mathbb{R}\) be defined as
\[
\calA_i(\u,\l;\v,\m) = \langle A_i \u, \v \rangle - \langle B_i \v, \l \rangle + \langle B_i \u, \m \rangle,
\]
we can rewrite equations \eqref{nonstab1}-\eqref{nonstab2} in compact form as: look for $(\u^i,\l^i) \in \calV_i$ such that for all $(\v^i,\m^i) \in \calV_i$ it holds
\begin{equation}\label{pb_cont_compact}
\calA_i(\u^i,\l^i;\v^i,\m^i) = \int_{\F_i} g^i \v^i + \int_{\Sqi} \ph \m^i.
\end{equation}
 The ideal residual based  stabilization as proposed by \cite{bertoluzza2022algebraic} consists in replacing, already at the continuous level, problem \eqref{pb_cont_compact} with
  the equivalent problem
 \begin{multline}
\calA_i(\u^i,\l^i;\v^i,\m^i) + \stabweight [\, A_i \uh^i  - \summi  B_{i,m}^T \lhim , t A_i \vh^i - \summi  B_{i,m}^T \mh^{i,m} ]_{H^{-1}_*(\Fi)} 
\\ = \int_{\F_i} g^i \v^i + \int_{\Sqi} \ph \m^i + \stabweight  [\, g^i , t A_i \vh^i - \summi  B_{i,m}^T \mh^{i,m} ]_{H^{-1}_*(\Fi)}.
\end{multline}
 where $\stabweight > 0$ is a (mesh independent) weight that must be less or equal than a constant $\stabweight_0$, and where $[\cdot,\cdot]_{H^{-1}_*(F_i)}$ denotes the scalar product for the space $H^{-1}_*(\Fi)$.
 We can write the stabilized problem in compact form as
  \[
 \calA_i(\u^i,\l^i;\v^i,\m^i) + \stabweight \calS_i(\u^i,\l^i;\v^i,\m^i)= \int_{\F_i} g^i \v^i + \int_{\Sqi} \ph \m^i + \stabweight \calG_i(\v^i,\m^i),
 \]
 where
\[
\calS_i(\u,\l;\v,\m) = [\, A_i \uh^i  - \summi  B_{i,m}^T \lhim , t A_i (\wh^i) - \summi  B_{i,m}^T \mh^{i,m} ]_{H^{-1}_*(\Fi)}, 
\]
and
\[
\calG_i(\v,\m) = [\, g^i , t A_i (\wh^i) - \summi  B_{i,m}^T \mh^{i,m} ]_{H^{-1}_*(\Fi)}.
\]
Different values of the parameter $t$, which can a priori be any real number but typically assumes values in $\{-1,0,1\}$, result in different versions of the stabilization.

Following the approach of \cite{bertoluzza2022algebraic}, at the discrete level, the $H^{-1}_*(\Fi)$ scalar product involved in the definition of the stabilizing term $\calS_i$ and $\calG_i$ is replaced by a projected $H^{-1}_*(\Fi)$ scalar product, realized algebraically with the aid of suitable auxiliary spaces.
More precisely, to realize the stabilization we need, for each fracture $\Fi$, an auxiliary subspace $\Wauxi \subset H^1_*(\Fi)$ such that
\begin{equation}\label{infsupaux}
\inf_{\lh \in \Lh^i} \sup_{\wh \in (V_\h^i + \Wauxi )} \frac{
\int_\Sqi \lh \wh 
}{\| \lh \|_{\L^i} \| \wh \|_{1,\Fi} } \gtrsim 1.
\end{equation}

We let $\{\auxb_k, \  k = 1, \cdots, N \}$ denote a basis for the auxiliary space $\Wauxi$ and we introduce the matrices $\matR$ and $\matS$ defined as
\begin{gather*}
	\matR = (r_{k\ell}),\ 	\matS = (s_{k\ell}), \qquad r_{k\ell} =\int_{\Fi} \nabla \auxb_\ell \cdot \nabla \auxb_k, \qquad \matS = \matR^{-1}.
\end{gather*}
The stabilization recipe proposed in  \cite{bertoluzza2022algebraic} consists in replacing the dual scalar product for $H^{-1}_*(\Fi)$ with a discrete scalar product of the form
\begin{gather}\label{abstractdualnorm}
	[f,g]_{-1,\Fi} = \sum_{\ell,k} s_{k\ell} \langle f, \auxb_\ell \rangle \langle g,\auxb_k \rangle,\qquad f,g \in H^{-1}_*(\Fi).
\end{gather}
The stabilizing terms are then defined as
\begin{gather}\label{defstabilizingterm}
	\calSnew_i(\u,\l;\v,\m) = [\, A_i \uh^i  - \summi  B_{i,m}^T \lhim , t A_i (\wh^i) - \summi  B_{i,m}^T \mh^{i,m} ]_{-1,\Fi}, \\ \calGnew_i(\v,\m) = [\, g^i , t A_i (\wh^i) - \summi  B_{i,m}^T \mh^{i,m} ]_{-1,\Fi}.
\end{gather}

The global discrete problem is 
\begin{pb}\label{pb:discrete-stabilized}
	Find $\uh = (\uh^i) \in V_\h$, $\lh = (\lhi ) \in \Lh$, $\phh \in \Ph_\h$ such that 
	\begin{multline}\label{stab1e2}
	\calA_i(\u^i,\l^i;\v^i,\m^i) + \stabweight \calSnew_i(\u^i,\l^i;\v^i,\m^i)\\= \int_{\F_i} g^i \v^i + \int_{\Sqi} \ph \m^i + \stabweight \calGnew_i(\v^i,\m^i), \qquad \forall (\vh^i,\lhi ) \in \calV_\h^i, \ \forall i,
\end{multline}
and 
\begin{gather}
		\sumi  \int_{\Sqi} \lhi  \psh = 0, \qquad \forall \psh \in \Phh.\label{stab3}
	\end{gather}
\end{pb}

\

We have the following theorem, which can be proved by essentially the same argument as for the proof of \cite[Theorem 3.1]{bertoluzza2022algebraic}.

\begin{thm}\label{thm:main}
Assume that the auxiliary spaces $\Wauxi$ satisfy \eqref{infsupaux}, as well as a Poincar\'e inequality of the form
\[
\| w \|_{0,\F_i} \lesssim | w |_{1,\F_i} \qquad \text{ for all } w \in \Wauxi.
\]
Then, there exists a constant $\stabweight_0 > 0$ such that, provided $\stabweight \leq \stabweight_0$, Problem \ref{pb:discrete-stabilized} is well posed. Moreover, the following error estimate holds
\[
\sumi  \| \u^i - \uh^i \|_{1,\F_i}^2 \lesssim
\inf_{\vh \in V_h} \sumi  \| \u^i - \vh^i \|_{1,\F_i}^2 + \inf_{\mh \in \Lh}  \| \l^i - \mh^i \|_{\L}^2 + \inf_{\psh \in \Phh} \| \ph - \psh \|_{\Ph}^2.
\]
\end{thm}

\newcommand{\KerB}{\mathcal{K}}
\newcommand{\Waux}{W_\h}
\begin{proof}
	We consider the following 
	auxiliary problem, equivalent to Problem \ref{pb:discrete-stabilized}: find $\u_\h = (\u^i_\h )_i \in V_\h, \huh = (\huh^i)_i \in \Waux = \prod_i \Wauxi, \lh =(\lhi)_i \in \Lh$, $\phh \in \Phh$ such that for all $\vh = (\vh^i)_i \in V_\h, \wh = (\huh^i)_i \in \Waux, \lh =(\lhi)_i \in \Lh$, $\psh \in \Phh$ it holds that
	\begin{gather}
\label{pbexp1}	\sumi 	\int_{\Fi} K_i \nabla \uh^i \cdot \nabla \vh^i - \sumi  \int_\Sqi \lhi  \vh^i = \sumi   \int_{\Fi} g^i \vh^i,  \\[1mm]
\label{pbexp2}	\sumi 	\int_{\Fi} K_i \nabla \uh^i \cdot \nabla \wh^i  + \frac 1 \stabweight \sumi  \int_{\Fi} \nabla \huh^i \cdot \nabla \wh^i - \sumi  \int_{\Sqi} \lhi  \wh^i =\sumi   \int_{\Fi} g^i \wh^i, \\[2mm]
\label{pbexp3}	\sumi 	\int_\Sqi (\uh^i + \huh^i) \mh^i = \sumi  \int_\Sqi \phh \mh^i, \\
	\sumi  \int_{\Sqi} \lhi  \psh = 0.\label{pbexp4}
	\end{gather}
It is not difficult to see that the solution of Problem \ref{pb:discrete-stabilized} can be obtained from the solution of (\ref{pbexp1}--\ref{pbexp4}) by static condensation of the auxiliary unknown $\huh$. Then, well posedness of the latter implies well posedness of the former.
	Thanks to \eqref{infsupbase}, in order to prove well posedness of  (\ref{pbexp1}--\ref{pbexp4})  it is sufficient to prove the well posedness of  (\ref{pbexp1}--\ref{pbexp3}) in the set $V_\h \times \Waux \times \KerB$  with
	\[
	\KerB = \{ (\lh^i)_i \in \Lh:
	\ \sumi  \int_{\Sqi} \l_h^i \psh = 0 \ \forall \psh \in \Phh
	\}.
	\]
	This is, in turn, also  a saddle point problem, where the primal space is $V_\h\times \Waux$ and the multiplier space is $\Lh$. By assumption, we have that the discrete spaces satisfy the inf-sup condition \eqref{infsupaux}. We then only need to prove the following ``ellipticity on the kernel'' property: for $(\uh, \huh)$ with 
	\begin{equation}\label{kernel}
		\sumi  \int_{\Sqi} (\uh^i + \huh^i) \mh^i = 0, \qquad \forall \mh = (\mh^i)_i \in \KerB 
	\end{equation}
	it holds that
	\[
	\int_{\Fi} K_i \nabla \uh^i \cdot \nabla (\uh^i + \huh^i) + \frac 1 \stabweight  \int_\Fi | \nabla \huh^i |^2 \gtrsim \| \uh^i \|_{1,\Fi}^2 + \| \huh^i \|_{1,\Fi}^2.
	\]
	
	To prove this, we start by observing that for all $(\v^i)_i \in \prod_i H^1_*(\F_i)$ with 
\begin{equation}\label{ortkerb}
\sumi  \int_{\Sqi} \u^i \l_\h^i = 0, \qquad \forall \lh = (\l^i_\h)_i  \in \KerB
\end{equation}
it holds that 
\begin{equation}\label{poincare}
\sumi  \| \u^i \|_{0,\F_i}^2 \lesssim \sumi  | \u^i |_{1,\F_i}^2.
\end{equation}
Indeed, \eqref{ortkerb} implies that functions which are constant on each fracture, are continuous, which, thanks to the Dirichlet boundary conditions, implies that they vanish.


Then, for $\u_\h,\huh$ satisfying \eqref{ortkerb}
 we have that
\[
\sumi  \| \uh + \huh \|^2_{0,\Fi} \lesssim \sumi  | \uh + \huh |^2_{1,\Fi} 
\]
Moreover, by the Poincar\'e  assumption on the space $\Wauxi$, we have that 
\[
\| \huh \|_{0,\Fi} \lesssim | \huh |_{1,\Fi},
\]
which implies that
\begin{multline*}
\sumi  \| \uh \|^2_{0,\Fi} \lesssim \sumi  \| \uh + \huh \|^2_{0,\Fi} + \sumi  \| \huh \|^2_{0,\Fi}\\
 \lesssim \sumi  | \uh + \huh |^2_{1,\Fi} + \sumi  | \huh |^2_{1,\Fi}  \lesssim \sumi  | \uh |^2_{1,\Fi}  + \sumi  | \huh |^2_{1,\Fi}. 
\end{multline*}
Now we have 
\begin{multline*}
\int_{\Fi} K_i \nabla \uh^i \cdot \nabla (\uh^i + \huh^i) + \frac 1 \stabweight  \int_\Fi | \nabla \huh^i |^2 \\
\geq | \sqrt{K_i} \uh^i |_{1,\Fi}^2 - \left| \int_\Fi \nabla \uh^i \cdot \nabla \huh^i \right| + \frac 1 \stabweight  | \huh^i |_{1,\Fi}^2 \\
\geq  | \sqrt{K_i} \uh^i |_{1,\Fi}^2 + \frac 1 \stabweight  | \huh^i |_{1,\Fi}^2 - \frac 1 2 | \sqrt{K_i} \uh^i |_{1,\Fi}^2 - C(K_i)  | \huh^i |_{1,\Fi}^2.
\end{multline*}
Then, if $\omega$ is sufficiently small, we get well posedness. The error estimate can be obtained by the  standard theory for saddle point problems, after observing that $\u$, $\widehat \u = 0$, $\l$ and $\ph $ satisfy (\ref{pbexp1}--\ref{pbexp4}).
\end{proof}

As we already mentioned, the solution of problem \eqref{pb:3fcont} typically satisfies $u^i \in H^{3/2-\varepsilon}(\F_i)$, $u^i \not \in H^{3/2}(F_i)$. Then we have the following Corollary
\begin{cor}
	For al $\varepsilon > 0$ it holds that
\[
\sumi  \| \u - \uh^i \|_{1,\F_i}^2 \lesssim C_\varepsilon
\h^{1/2-\varepsilon}.
\]
\end{cor}

\NOTE{$ f \in W^{1,p} $, $ p > N$ continuous, $f \in W^{1,p}$, $ p < N$ discontinuous.
}

\begin{rem}
	The approach that we present can be applied  to other discretizations, including higher order ones. However, as the solution of the problem we consider is generally not regular, as it presents jumps in the conormal component of the gradient along traces, we cannot generally expect a convergence of order greater than one half, an order one discretization is in this framework sufficient. We also observe that, as far as the choice of the multiplier space $\Lh$ is concerned, it would also possible to resort to discontinuous piecewise constant discretization on a dual mesh to the one underlying the space. 
\end{rem}

\newcommand{\Fim}{F_{i,m}}
\newcommand{\dist}{\mathrm{dist}}

\section{Construction of the stabilizing term}
\label{sec:StabTerm}

We focus now on the concrete construction of the stabilizing bilinear forms $\calS_i$, the first step of which is the choice of the auxiliary spaces.
Assumption \eqref{separation} allows us to define $\Wauxi$ as the direct sum of subspaces $\Wauxim$ supported away from each other, each in a neighborhood of the corresponding trace $S_m$.  The double sum $\sum_{\ell,k}$ in \eqref{abstractdualnorm} can then be split as the $\sum_{S_m \in \Ti} \sum_{\ell,k}$ where, once $m$ is fixed, $\ell$ and $k$ are indexes of the basis functions for the local auxiliary space $\Wauxim$:
\[
	[f,g]_{-1,\Fi}  = \sum_{S_m \in \Ti} \sum_{\ell,k} s_{k,\ell} \langle f, \auxb^m_\ell \rangle \langle g,\auxb^m_k \rangle =  \sum_{S_m \in \Ti} [f,g]_{-1,i,m},
\]
where the functions $\auxb_\ell^m$, $\ell =1, \cdots, \dim(\Wauxim)$, form a basis for $\Wauxim$. Remark that the resulting matrices $\matR$ and $\matS = \matR^{-1}$ are block diagonal, with  blocks $\matS^m$ and $\matR^m$  corresponding to the trace $S_m$).

\

In order to construct the  spaces $\Wauxim$ so that  the inf-sup condition \eqref{infsupaux} is satisfied, for each trace $S_m \in \Ti$ we introduce disjoint  subdomains $\Fim \subseteq \F_i$, each defined as the union of triangles $T$ of a local shape regular mesh $\Taux^{i,m}$, with meshsize $\gtrsim \h$.  We assume that the set of interior nodes of $\Taux^{i,m}$ coincides with the set of nodes of $\TSq^{i,m}$ which are interior to $\Fi$.  We let $\Wauxim$ be the space of continuous piecewise linears on $\Taux^{i,m}$, vanishing at all nodes which are not on $S_m$, as well as at nodes on the Dirichlet boundary $\bOD$, if any (see Figure \ref{fig:1}). Thanks to Assumption \eqref{separation} it is indeed possible to construct local auxiliary meshes $\Taux^{i,m}$ in such a way that the corresponding local subdomains are disjoint.

\begin{figure}
	\includegraphics*[height=6cm]{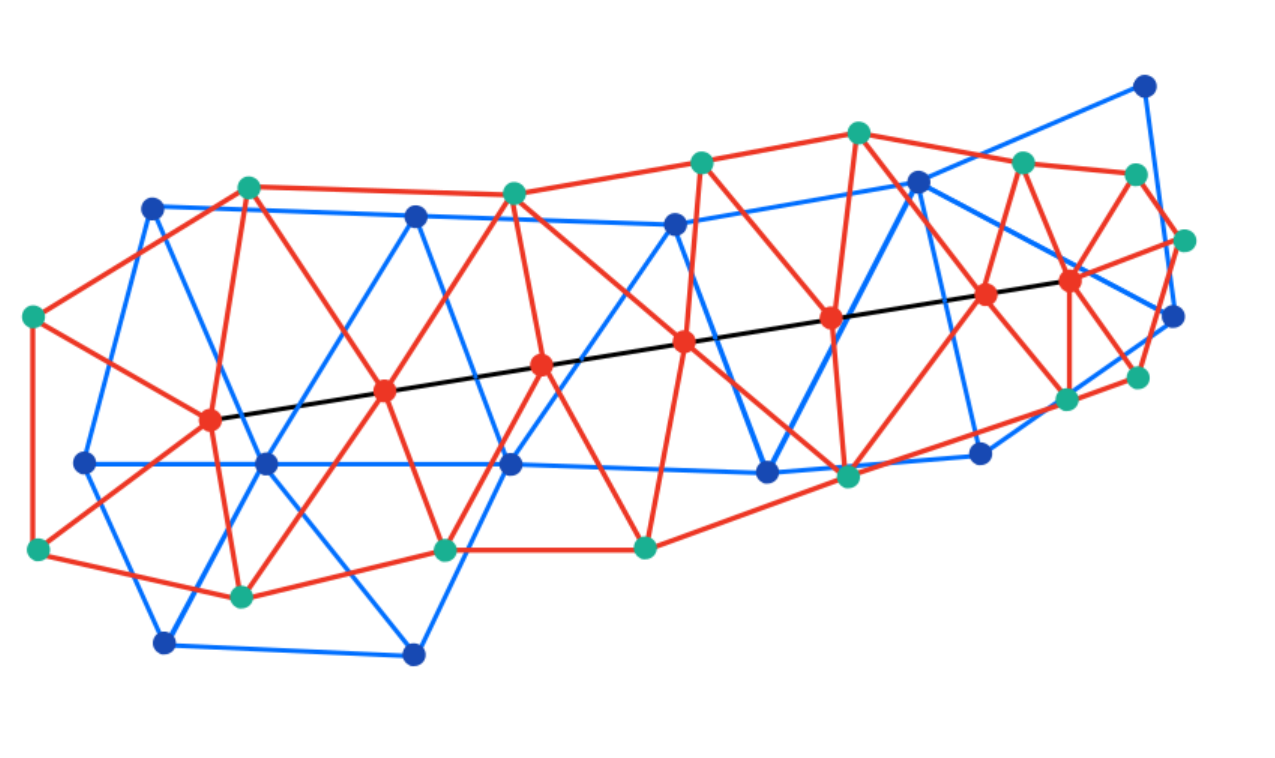}
	\caption{
		Triangulations $\Th$ (blue) and $\Taux$ (red). $\Fim$ is the union of the red triangles.
	The red dots mark the degrees of freedom for $\Wauxim$ (which coincide with the dofs for $\Lh^{i,m}$), while in the nodes marked with the green dots  $\wh = 0$. }\label{fig:1}
\end{figure}

Observe that, by definition, we have that a) $\Wauxim|_{S_m} = \Phh|_{S_m}$, and b)  for $m'\not= m$ the supports of functions in $\Wauxim$ and $W^{i,m'}_{\h}$ do no overlap.    We define the auxiliary space $\Wauxi$ as the direct sum of the local spaces $\Wauxim$.
We have the following lemma.

\newcommand{\lho}{\l_\h^0}
\newcommand{\lhperp}{\l_\h^\perp}
\newcommand{\Lho}{\L_\h^0}
\newcommand{\ho}{\h_0}

\begin{lem}
	The inf-sup condition \eqref{infsupaux} holds.
\end{lem}

\begin{proof}
	By standard arguments, we can prove that there exists a subspace $\Lho  \subset \Lh$ of mesh size $\simeq \h$ such that
	\begin{equation}\label{infsupcoarse}
		\inf_{\lh \in \Lho} \sup_{\vh \in V_\h^i} \frac { \int_{\Sqi} \lh \vh}{\| \lh \|_{\L^i} \| w_h \|_{1,F_i}} \gtrsim 1.
	\end{equation}
	
	Indeed, let $\Lho $ denote the subspace corresponding to a coarser embedded  mesh of meshsize $\ho$, and let $\lh \in \Lho $. For a suitable $w^{\lh }\in V^i$ and for $w_\delta \in V_\h^i$ a suitable quasi interpolant of $w^\lh$  we have
	\begin{multline*}
		\| \lh \|_{\L^i} = \sup_{w \in V^i} \frac {\langle \lh, w \rangle} {\| w \|_{1,\Fi}} \lesssim \frac {\langle \lh, w^\lh \rangle} {\| w^\lh \|_{1,\Fi}} =
		\frac {\langle \lh, w^\lh - \vh \rangle} {\| w^\lh \|_{1,\Fi}} + \frac {\langle \lh, \vh \rangle} {\| w^\lh \|_{1,\Fi}}\\
		\lesssim 
		\frac {\| \lh \|_{0,\Sqi} \|  w^\lh- \vh \|_{0,\Sqi}} {\| w^\lh \|_{1,\Fi}} + \frac {\langle \lh, \vh \rangle} {\| w^\lh \|_{1,\Fi}} \\
		\lesssim \frac{\ho^{-1/2} \| \lh \|_{\L^i} \h^{1/2} \|  w^\lh  \|_{1,\Fi}} {\| w \|_{1,\Fi}} + \frac {\langle \lh, \vh \rangle} {\| w^\lh \|_{1,\Fi}}.
	\end{multline*}
	This implies that for some constant $C$ we have
	\[
	\| \lh \|_{\L^i} (1 - C \frac{\h^{1/2}}{\ho^{1/2}}) \lesssim \frac {\langle \lh, \vh \rangle} {\| w^\lh \|_{1,\Fi}}.
	\]
	If we choose $\ho \sim K \h$ with $K$ sufficiently large we then obtain \eqref{infsupcoarse}.
	
	\
	
	\newcommand{\wo}{w^0}		
	
	Let now $\lh \in \Lh^i$. We can split $\lh$ as $\lh = \lho  + \lhperp $ with $\lho  \in \Lho $ and $\lhperp $ satisfying
	\[
	\int_{\Sqi} \lhperp  \mu_h^0 = 0, \qquad \text{for all }\mu_h^0 \in \Lho .
	\]
	We let $\wo \in V_\h^i$ with $\| \wo \|_{1,\Fi} \simeq \| \lho \|_{\L^i}$ be such that $\langle \lho , \wo \rangle \geq \| \lho \|_{\L^i}^2$.
	The high frequency component $\lhperp $ satisfies a direct inequality of the form
	\[
	\| \lhperp  \|_{\L^i} \lesssim \ho^{1/2} \| \lhperp  \|_{0,\Sqi} \lesssim \ho^{1/2} \frac{ \langle \lhperp , \wh \rangle}{\| \wh \|_{0,\Sqi}},
	\]
	where $\wh \in \Wauxi$ is the element coinciding with $\lhperp $ on $\Sqi$, and vanishing on all other nodes of the auxiliary mesh $\Taux$. Letting $\Nodesaux^{m,i}$ denote the set of nodes of the local  auxiliary mesh $\Taux^{m,i}$, and using the equivalence of the $L^2$ norms of finite element functions with the weighted euclidean norm of their nodal values, I have 
	\begin{equation}\label{equiv1}
		\| \wh \|_{1,\Fi}^2 = \sum_{T \in \cup_m \Taux^{m,i}} \| \wh \|_{1,T}^2 \simeq \sum_{x\in\cup_{m}\Nodesaux^{m,i}} | \wh(x) |^2 \simeq \h^{-1} \| \wh \|_{0,\Sqi}^2,
	\end{equation}
	so that I can write
	\[
	\| \lhperp  \|_{\L^i} \lesssim \frac{\h_0^{1/2}}{\h^{1/2}} \frac{ \langle \lhperp , \wh \rangle}{\| \wh \|_{1,F_i}} \lesssim 
	\frac{ \langle \lhperp , \wh \rangle}{\| \wh \|_{1,F_i}}.
	\]

	As, by construction, $\wh|_{\Sqi} \perp \lho $, we have
	\begin{multline*}
		\langle \lho  + \lhperp , \vh + \alpha \wh \rangle 
		= \langle \lho , \wo \rangle +  \langle \lhperp , \wo \rangle  + \alpha \langle \lhperp ,\wh \rangle 
		\\	\geq \| \lho  \|^2_{\L^i} + \alpha \| \lhperp  \|_{\L^i}^2 - C \| \lho \|_{\L^i}  \| \lhperp \|_{\L^i}.
	\end{multline*}
	Provided $\alpha$ is large enough, we then obtain 
	\[
	\langle \lho  + \lhperp , \vh + \alpha \wh \rangle \gtrsim  \| \lho  \|_{\L^i}^2 +  \| \lhperp \|^2_{\L^i}.
	\]
	We conclude by observing that, by triangle inequality, $\| \lh \|_{\L^i} \leq \| \lho \|_{\L^i} + \| \lhperp \|_{\L^i}$, and that we have
	\[
	\| \wo + \alpha \wh \|_{1,\Fi} \leq \| w_\delta \|_{1,F_i} + \alpha \| \wh \|_{1,F_1}  \lesssim \| \lho  \|_{\L^i} +  \| \lhperp \|_{\L^i},
	\]
	which concludes the proof.
\end{proof}

In order to implement the proposed stabilization, we next need to evaluate the stabilization terms. 
Looking at the definition \eqref{abstractdualnorm}, we see that this reduces to evaluating the action on  elements $\wh \in \Wauxi$, of the $H^{-1}_*(\Fi)$ functional $A_i \uh -\sum_{S_m\in \Ti} B^T_{i.m} \lhim  $, where $\uh \in V_\h^i$ and $\lhim  \in \Lh^{i,m}$ are piecewise linear finite element functions, as defined in  \eqref{defVh},\eqref{defPh},\eqref{eq:defLh}. 
 We are able to evaluate such an action of as follows:
\begin{multline}
\langle A_i \uh  , \wh \rangle  =
\sum_{T \in \Thi} \int_T K_i \nabla \uh^i\, \cdot\, \nabla \wh  
=
\sum_{T \in \Thi} \left(\int_{\bK} K_i \nabla \uh^i \cdot \nK \wh 
\right)
 \\ = \sum_{e \in \Edges} \int_{e} \jump{ K_i \nabla \uh^i \cdot \norm}  \wh = \sum_{e \in \Edges} \jump{ K_i \nabla \uh^i \cdot \norm} |_e \int_{e} \wh.
\end{multline}
In particular, for $\wh = \auxb_\ell^m$, $\auxb_\ell^m$ being one of the nodal basis functions for $\Wauxim$, we get
\begin{equation}\label{coefsevalAu}
	\langle A_i \uh  , \auxb^m_\ell \rangle
	= \sum_{e \in \Edges} \weight_e^\ell \jump{ K_i \nabla \uh^i \cdot \norm} |_e .
\end{equation}
where $\weight_e^\ell = \int_e \auxb^m_\ell$, and where we exploited the fact that, on each edge $e$ of the triangulation $\Thi$ the flux $\jump{K_i \nabla \u_\h^i \cdot \nu}$ is a constant.

 As far as $B_{i}$ is concerned, we have
\begin{equation}\label{coefsevalBTl}
	\langle B_i \lh  , \wh \rangle = 	\langle  \sum_{S_m\in \Ti} B_{i,m}^T \lhim  , \wh \rangle  =
\sum_{S_m\in \Ti}\int_{S_m} \lhim  \wh.
\end{equation}
In the above expression, for which, until now, we did not exploit any information on the function $\wh$, the sum spans over the triangles and edges of the mesh $\Thi$, where the functions $h^i_\h$ and $w^i_\h$ ``live''. The stabilization term is obtained by using expressions \eqref{coefsevalAu} and \eqref{coefsevalBTl} in \eqref{defstabilizingterm}.

\section{A new mesh dependent stabilization} \label{sec:stab_meshdependent}
We now want to define, starting from  the definition \eqref{defstabilizingterm} of the bilinear form $\calS_i$,
 a simpler stabilization term.
 Assuming that  $\Wauxim$ is defined as in Section \ref{sec:StabTerm}, so that $\Wauxim|_{S_m} = \L_\h^{i,m}$ and the degrees of freedom for $\Wauxim$ and for $\L_\h^{i,m}$  coincide, 
we propose to replace, in the definition \eqref{abstractdualnorm}, the matrix $\matS$, which is the inverse of a stiffness matrix, with a suitably scaled inverse of an $L^2(S_m)$ mass matrix.

\

To this aim we choose $\hBasis$ to be the nodal basis for $\Wauxim$, and we introduce a second basis $\tBasis$ for the space $ \Lh^{i,m} = \Wauxim|_{S_m}$. This second basis will only be needed ``on paper'',  for deriving the new stabilization term, and will not actually need to  be implemented. More precisely, we let  \[\tBasis = \{\tw^m_k, \cdots k = 1, \cdots, N_m  \}\] be the basis for $\Lh^{i,m}$ which is bi-orthogonal to the basis $\hBasis = \{\hw_k|_{S_m}, \ k = 1, \cdots, N_m  \}$: the elements of $\tBasis$ are defined by the biorthogonality condition
\begin{equation}
\int_{S_m} \tw^m_\ell \hw^m_n = \delta_{n\ell}, \qquad n,\ell =1,\cdots,N_m.
\end{equation}
The restrictions to $\T_m$ of the functions $\hw_\ell$  coincide with the nodal basis for the space $\Lh^{i,m}$.  The 
$\tw_k$ are instead non local piecewise linear functions on $\T_m$ which, as already pointed out, are only used for the analysis and whose actual construction is never needed in practice.
We observe that we have
\[
\| \hw_\ell \|_{0,S_m} \simeq \sqrt{\h}, \qquad \| \hw_\ell \|_{0,\infty S_m} \simeq 1, \qquad \| \tw_\ell \|_{0,S_m} \simeq \frac1{\sqrt{\h}}, \qquad \| \tw_\ell \|_{0,\infty,S_m} \simeq \frac 1 \h.
\]

 We have the following expansion for the elements  $\lh \in \Lh^{i,m}$:
\begin{equation}\label{expansiontilde}
 \lh =   \sum_{\ell = 1}^{N_m} \left(\int_{S_m} \lh \hw^m_\ell \right) \tw^m_\ell.
\end{equation}
As $\tBasis$ is biorthogonal to  the nodal basis, it is not difficult to check that we then have the identity
\[
\int_{S_m} \lh \tw^m_\ell = \lh(x_\ell), \qquad \forall \lh \in \Lh^{i,m}.
\]
Moreover, as by construction  the elements of $\Wauxim$ are uniquely identified by their trace on $S_m$, we also have the identity
\begin{equation}
	\hw = \sum_{\ell = 1}^{N_m} \left(\int_{S_m} \hw \tw^m_\ell \right) \hw^m_\ell.
\end{equation}
Given $\lh \in \Lh^{i,m}$, we let $\hw(\lh) \in \Wauxim$ be defined as
\[
\hw(\lh) = \sum_{\ell = 1}^{N_m} \left(\int_{S_m} \lh \tw^m_\ell \right)\hw^m_\ell.
\]
It is not difficult to check that
\[ \hw(\lh)|_{S_m} = \lh.\]

%


To build the new stabilization term, we replace the matrix $\matS^m$ in the definition of the bilinear form $\stab\cdot\cdot$ with  the scaled mass matrix $\h \, \widetilde \matM^m$  on $\Lh^{i,m}$, computed w.r. to the $\tw_k$ basis:
\[
\widetilde  {\mathsf{m}}_{k\ell}  = \int_{S_m} \tw^m_k \tw^m_\ell.
\]
Observe that  $\tM = (\widetilde{\mathsf{m}}_{k\ell}) = (\matM^m)^{-1}$ coincides with the mass matrix for  $\Lh^{i,m}$, computed w.r. to the $\tw_k$ basis. 
We then set
\begin{multline*}
\stab{f}{g} = \sum_{\ell,k} \h\, \widetilde {\mathsf{m}}_{k\ell} \langle f, \hw^m_\ell \rangle \langle g,\hw^m_k \rangle  = \sum_{\ell,k} \h\, \int_{S_m} \tw^m_k \tw^m_\ell \langle f, \hw^m_\ell \rangle \langle g,\hw^m_k \rangle \\=
\h \int_{S_m} \Big( \sum_k \langle g,\hw^m_k \rangle \tw^m_k \Big)
\Big( \sum_\ell \langle f,\hw^m_\ell \rangle \tw^m_\ell \Big).
\end{multline*}

Letting $f$ and $g$ be replaced by the action of the operators $A$ and $B$, we obtain the building blocks for the new stabilization.

More in detail, we have 
\begin{multline*}
	\stab{A\uh}{A\vh} \\=
	\h  \sum_k \sum_\ell  
	\Big(\sum_{e \in \Edges} \weight_e^k \jump{ K_i \nabla \uh^i \cdot \norm} |_e\Big)
 \Big(
	\sum_{e \in \Edges} \weight_e^\ell \jump{ K_i \nabla \vh^i \cdot \norm} |_e
\Big) 	
	 \int_{S_m}  \tw^m_k(x) 
	 \tw^m_\ell(x)  \,dx \\
=	 	\h  
	 	\sum_{e \in \Edges}
	 		\sum_{e' \in \Edges}
	 	 \jump{ K_i \nabla \vh^i \cdot \norm} |_{e'}
	  \jump{ K_i \nabla \uh^i \cdot \norm} |_e 
	 	\sum_k \sum_\ell  
	  \weight_e^k 
\weight_{e'}^\ell 
	 \widetilde{ \mathsf{m}}_{\ell k}.
\end{multline*}
and 
\begin{multline*}
	\stab{B^T \mh}   {B^T \lh} \\ =  \h
	\int_{S_m}  \, dx   
\Big(\sum_\ell \tw^m_\ell(x)\int_{S_m}\,dy \mh^i(y) \hw_\ell(y) \Big)
	\Big(\sum_k \tw^m_k(x)\int_{S_m}\,dy \lhi (y) \hw_k(y) \Big)\\
 = \h \int_{S_m} \lh \mh,
\end{multline*}
as well as 
\begin{multline*}
\stab{A \uh } {B^T \mh} \\= \h
\int_{S_m}  \, dx \Big(\sum_{\ell} \ \tw^m_\ell(x) 
\sum_{e \in \Edges} \weight_e^\ell \jump{ K_i\nabla \uh^i \cdot \norm}  
\Big) \Big(\sum_k \tw^m_k(x)\int_{S_m}\,dy \mh^i (y) \hw_k(y)\Big)\\
=  \h\sum_{\ell} \Big( \sum_{e \in \Edges} \weight_e^\ell \jump{ K_i \nabla \uh^i \cdot \norm}  \Big)
\int_{S_m}  \,dx   \tw_\ell(x) 
 \mh^i(x) \\ =  \h\sum_{\ell}  \Big( \sum_{e \in \Edges} \weight_e^\ell \jump{K_i \nabla \uh^i \cdot \norm}  \Big)
\mh^i(x_\ell).
\end{multline*}
Finally, we have
\begin{multline*}
	\stab{g}{B^T \mh} = \h	\int_{S_m}\, dx \mh^i(x) \Big( \sum_{\ell} 
	\tw^m_\ell(x) \Big(\int_{\Fim} g(y)  \hw^m_\ell(y) \,dy\Big)
\Big)\\
	= \h \int_{\Fim}   g(y)\Big( \sum_{\ell} \Big(\int_{S_m}\, dx \mh^i(x) 
	\tw^m_\ell(x) \Big) \hw^m_\ell(y)
\Big)  \,dy = \h \int_{\Fim}   g(y) \hExt{\mh^i}  \,dy ,
\end{multline*}
where $\hExt{\lh} \in \Wauxim$ is the function in $\Wauxim$ that coincides with $\lh$ on $S_m$.

Letting 
\[
[\cdot,\cdot]_{-1,i,\h} =  \sum_{m\in \Ti} \stab\cdot\cdot,
\]
we can consider the following  stabilized discrete problem 
\begin{pb}\label{pb:3bBHstab}
	Find $\uh = (\uh^i)_i \in V_h$, $\lh = (\lh^i)_i \in \Lh$, $\phh \in \Phh$ such that
\begin{gather*}
\hskip-4cm	\int_{\Fi} K_i \nabla \uh^i \cdot \nabla \vh^i - \int_\Sqi \lhi  \vh^i + \alpha t 
\stabi{A \uh^i - B^T \lhi }{A \vh^i}
\\ \hskip6cm= \int_{\Fi} g^i \vh^i + \alpha t \stabi {g^i} {A \vh^i} \qquad \forall \vh^i \in V_\h^i, \ \forall i\\
	\int_\Sqi \uh^i \mh^i - \int_\Sqi \phh \mh^i  - \alpha \stabi{A\uh - B^T \lh}{B^T \mh} = \alpha \stabi{g^i}{B^T \lh g^i}, \qquad \forall \mh^i \in \Lh^i, \ \forall i \\
	\sumi  \int_{\Sqi} \lhi  \psh = 0, \qquad \forall \psh \in \Phh.
\end{gather*}
\end{pb}

\newcommand{\uI}{\u_\pi}
\newcommand{\lI}{\l_\pi}
\newcommand{\phI}{\phh^I}

We have the following theorem.
\begin{thm}\label{thm:5.2} There exists $\alpha_0$ such that for all $\alpha < \alpha_0$ Problem \ref{pb:3bBHstab} is well posed, and the following estimate holds
	\begin{multline}
	 \normoneh{\u - \uh}
	 \lesssim \inf_{\uI \in V_\h} \normoneh{\u - \uI}+ \inf_{\mh \in \Lh} \| \l - \mh  \|_{\L} + \inf_{\psh \in \Phh} \| \ph - \psh \|^2_{\Ph}.
	\end{multline}
with
\[
\normoneh{\u} = 	\sum_{i=1}^L | \uh - \u |^2_{1,\Fi} 
+  \summ | S_m | \left| 
\fint_{S_m} \scjump{\uh} \right|^2,
\]
 where for $S_m = \Fi \cap F_j$,  $i < j$, we set $\scjump{\uh} = h_i - h_j$.  
\end{thm}

Once again, the $H^{3/2-\varepsilon}(F_i)$ smoothness of $\u$ yields the following corollary.
\begin{cor}
	For al $\varepsilon > 0$ it holds that
	\[
\normoneh{u - \uh}	  \lesssim C_\varepsilon
	\h^{1/2-\varepsilon}.
	\]
\end{cor}

\section{Proof of theorem \ref{thm:5.2}}
We  start by rewriting  Problem \ref{thm:5.2}  as the saddle point problem, namely
\begin{gather*}
	\A(\uh,\lh;\vh,\mh) + \B(\vh,\mh;\ph) = \F(\vh,\mh)\\
	\B(\uh,\lh;\psh) = 0,
\end{gather*}
where $\A: \calV_\h \times \calV_\h \to \mathbb{R}$ and $\B: \calV_\h \times \Ph_\h \to \mathbb{R}$ are defined as
\begin{multline*}
	\A(\uh,\lh;\vh,\mh) = 
	\sumi   	\int_{\Fi} K_i \nabla \uh^i \cdot \nabla \vh^i 
	-  \sumi    \int_\Sqi \lhi  \vh^i 
	+  \sumi   	\int_\Sqi \uh^i \mh^i  \\
	+  \alpha \glstab{A\uh - B^T \lh}{t A\vh - B^T \mh},
\end{multline*}
\begin{equation*}
	\B(\uh,\lh;\psh) = 	\sumi  \int_{\Sqi} \lhi  \psh.
\end{equation*}

We then need to prove an inf-sup condition for $\B$ and a coercivity, or an inf-sup condition, for $\A$ on $\ker \B \subseteq \calV_\h$, where we use the notation $\ker \B$ to denote the  kernel of linear operator corresponding to the bilinear form $\B$,  which is the space of functions $(\uh,\lh)$ such that
\(
\sumi  \lhi  
\) is orthogonal to $\Phh$. 

\

We introduce mesh dependent norms on $\Lh$ and $\Phh$:
	\[
\| \l  \|_{-1/2,\h}^2 = \sumi  \h_i \| \lhi  \|^2_{0,\Sq_i}, \qquad \| \ph \|^2_{1/2,\h} = \sumi  \h_i^{-1} \| \ph \|^2_{0,\Sqi}
\]

The bilinear form $\B$ is continuous with respect to such norms. Indeed we have
\begin{multline}
\B(u,\l ;\ph ) \lesssim \sumi  \sum_{m \in \Ti} \| \l _i \|_{0,S_m} \| \ph  \|_{0,S_m} = \sumi  \sum_{m \in \Ti}  \h^{1/2} 
\| \l  _i \|_{0,S_m} \h^{-1/2} \| \ph  \|_{0,S_m}.
\end{multline}
Moreover by the definition of $[\cdot,\cdot]_{i,m,\h}$ we have that
\begin{multline*}
	\stab{A \uh} {B^T \lh}   = \h \int_{S_m} \l_\h^i(x) \left(
	\sum_{\ell} \ \tw_\ell(x) 
	\sum_{e \in \Edges} \weight_e^\ell \jump{ K_i \nabla \uh^i \cdot \norm}  
	\right)\\ \lesssim \h \| \l_\h^i \|_{0,S_m} \| \sum_{\ell} \ \tw_\ell(x) 
	\sum_{e \in \Edges} \weight_e^\ell \jump{ K_i \nabla \uh^i \cdot \norm}  \|_{0,S_m}\\
	\lesssim \h^{1/2} \| \lhi \|_{0,S_m} \sqrt{ \sum_k  \sum_{e\in \Edges} | \theta^\ell_e \jump{K_i\nabla \uh^i \cdot \norm} | ^2}\\
	\lesssim \h^{1/2} \| \lhi  \|_{0,S_m}  \sqrt{\h   \sum_{e\in \Edges} \|  \jump{K_i\nabla \uh^i \cdot \norm} \|_{0,e} ^2} \lesssim \h^{1/2} \| \lhi  \|_{0,S_m}   | \uh |_{1,\Fi}.
\end{multline*}
Analogously, we see that
\begin{multline*}
\stab{A \uh}{A \vh}  =	 	\h  
\sum_{e \in \Edges}
\sum_{e' \in \Edges}
\jump{ K_i \nabla \vh^i \cdot \norm} |_{e'}
\jump{ K_i \nabla \uh^i \cdot \norm} |_e 
\sum_k \sum_\ell  
\underbrace{\weight_e^k 
\weight_{e'}^\ell 
\widetilde{ \mathsf{m}}_{\ell k}}_{\lesssim \h} 
\\
\lesssim 
\h
\| \nabla \uh^i \|_{\infty,\Fi} 
\| \nabla \vh^i \|_{\infty,\Fi} \lesssim | \uh^i |_{1,\Fi} | \vh^i |_{1,\Fi},
\end{multline*}
where we used that, for a given $k$, the number of edges $e$ for which $\theta^k_e \not = 0$ is bounded independently of $\h$. 
Then, collecting the contributions of the different fractures and traces  we have that
\[
\A(\uh,\lh;\vh,\mh) \lesssim (\normoneh{\uh}  + \| \lh \|_{-1/2,\h}) (\normoneh{ \vh} + \| \mh \|_{-1/2,\h}).
\]

\NOTE{We also  observe that
\begin{multline*} 
	\stab{g}{B^T \lh} \leq \h \| \lh \|_{0,S_m} \| \sum_\ell \left(\int_{\Fim} g(y)  \hw_\ell(y) \,dy\right) \tw_\ell \|_{0,S_m}  \\
	\lesssim \h^{1/2} \| \lh \|_{0,S_m}  \sqrt{
		\sum_\ell \left(\int_{\Fim} g(y)  \hw_\ell(y) \,dy\right)^2
	}\\
	\lesssim \h^{1/2} \| \lh \|_{0,S_m}  \sqrt{
		| \Fim|	\sum_\ell \int_{\Fim} | g(y)  \hw_\ell(y)|^2 \,dy}\\
	\lesssim \h \| \lh \|_{0,S_m} 
	\sqrt{ \sum_\ell  \|  g \|^2_{0,\kappa_\ell}  } \lesssim \h \| \lh \|_{0,S_m} \| g \|_{0,\Fim},
\end{multline*}
where $\kappa_\ell = \supp \hw_\ell$, which gives us (part of the) continuity of the right hand side
}
\

The inf-sup condition for $\B$ with respect to the mesh dependent norms is also quite straightforward. We have 
\begin{multline*}
	\inf_{\phh \in \Ph} \sup_{\lh \in \Lh} \frac{\sumi  \int_{\Sqi}  \lhi \phh}{\| \lh \|_{-1/2,\h} \| \phh \|_{1/2,\h} } =
\inf_{\phh \in \Ph} \sup_{\lh \in \Lh} \frac{\sumi  \int_{\Sqi} (\h_i^{1/2} \lhi) (\h_i^{-1/2 }\phh)}{\| \h^{1/2} \lh \|_{0,\Sq} \| h^{-1/2} \phh \|_{0,\Sq} } \\
\geq 
\inf_{\phh \in \Ph} \sup_{\mh \in \Lh} \frac{\sumi  \int_{\Sqi}  \mh^i (\h_i^{-1/2 }\phh)}{\| \mh \|_{0,\Sq} \| \h_i^{-1/2}\phh \|_{0,\Sq} }  \gtrsim
 1.
\end{multline*}

\subsection{Inf-sup for  $\A$ on $\ker\B$} We have the following Lemma
\begin{lem}\label{infsupkerB}
There exists a $\alpha_0$ such that, provided $\alpha < \alpha_0$ it holds that for all $\uh, \lh$ there exists $\mu_h$ with
$\| \mh \|_{-1/2,\h} \lesssim\normoneh{ \uh } + \| \lh \|_{-1/2,\h}$, such that
	\[
	\A(\uh,\lh; \uh, \mh) \geq c^\sharp (\normoneh{\uh} + \alpha \| \lh \|_{-1/2,\h}).
	\]
\end{lem}
\begin{proof} 
We can write
\begin{multline*}
	\A(\uh,\lh;\uh,\lh) = 
	\sumi   K_i | \uh^i |_{1,\Fi}^2 
	+  \alpha \glstab{A\uh - B^T \lh}{t A\uh - B^T \lh} \\=
		\sumi  K_i  \left( | \uh |_{1,\Fi}^2 + 
		\alpha \stabi{B^T \lh}{B^T \lh} \right)
		\\ - (1 + t) \alpha 	  \glstab{A \uh } { B^T \lh} + \alpha t   \glstab{A \uh} { A \uh}\\
		\geq
	\sumi \left(  K_i	| \uh |_{1,\Fi}^2 	+ \frac{\alpha}2 \glstab{B^T \lhi }{B^T \lhi } - \frac \alpha 2 
	(1 + 3|t| + | t |^2) \glstab{A \uh^i}{A \uh^i}  \right) \\
		 \geq
		\sumi  \left(	| \uh |_{1,\Fi}^2 	+ \frac\alpha 2 \| \lh \|^2_{-1/2,\h} - C(t) \alpha \sumi  K_i | \uh |_{1,\Fi}^2  \right).
\end{multline*}

\NOTE{
\begin{multline*}
\alpha[x,x] - (1+t) \alpha [ y, x] + \alpha t [ y,y]
\geq \alpha [x,x] - (1+t) \alpha \sqrt{[x,x]}\sqrt{[y,y]} + \alpha t [y,y] \\
\geq
 \alpha [x,x] - (1+| t |) \alpha \sqrt{[x,x]}\sqrt{[y,y]} - \alpha | t | [y,y] \gtrsim
 \alpha [x,x] -  \frac \alpha 2 {[x,x]} -
  (1+| t |)^2   \frac \alpha 2 [y,y] - \alpha | t | [y,y]\\
  \frac \alpha 2 [x,x] - \alpha(\frac 12 + \frac32| t | + | t |^2) [y,y].
\end{multline*}

}

Choosing $\alpha$ sufficiently small, this gives us 
\begin{equation}\label{coercnonconst}
	\A(\uh,\lh;\uh,\lh) \geq  \ccoerca	\Big(\sumi  K_i | \uh |_{1,\Fi}^2 	+ \alpha  \| \lh \|^2_{-1/2,\h}\Big),
\end{equation}
which allows to control $\lh$ and $\uh$, up to constants in the fractures. To control the constant, we need to test the bilinear form $\A$  with a suitably chosen constant multiplier. 
\[
	\A(\uh,\lh;0,\mh) =    \sumi   	\int_\Sqi \uh^i \mh^i  
	-  \alpha \glstab{A\uh - B^T \lh}{ B^T \mh}
\]

Using $+$ and $-$ as labels for the two fractures whose intersection is a given trace,  on $S_m = \F^+ \cap \F^-$ we set $\mh^+ = \fint_{S_m} (\uh^+ - \uh^-)$, 
$\mh^- = -\fint (\uh^+ - \uh^-)$.  Observe that $(0,\mh) \in \ker\B$.
We have:
\begin{multline*}
	\A(\uh,\lh;0,\mh) =    \sumi   	\sum_{m \in \Ti} \int_{S_m} \uh^i \mh^i  
-  \alpha \glstab{A\uh - B^T \lh}{ B^T \mh}\\ =
\summ \sum_{i \in \Fm} \int_{S_m} \uh^i \mh^i  
-  \alpha \glstab{A\uh - B^T \lh}{ B^T \mh}
\end{multline*}
We have
\begin{multline*}
\sum_{i \in \Fm} \int_{S_m} \uh^i \mh^i   = \int_{S_m} \uh^+ \mh^+ + \int_{S_m} \uh^- \mh^- = \int_{S_m} (\uh^+  -  \uh^-) \mh^+ \\=
 \int_{S_m} (\uh^+  -  \uh^-) \fint_{S_m} (\uh^+ - \uh^-) = | S_m | \left| \fint_{S_m} \scjump{\uh} \right|^2
\end{multline*}
Moreover we have that
\[\| \mh \|_{-1/2,\h} =
\summ \h \| \mh \|_{0,S_m}^2 = \summ \h | S_m | \left| \fint_{S_m} \scjump{\uh} \right|^2
\]
Then
\begin{multline*}
\A(\uh,\lh;0,\mh) = \summ | S_m | \left| 
\fint_{S_m} \scjump{\uh} \right|^2
- \alpha\glstab{A\uh - B^T\lh}{B^T \mh} \\
\geq
 \summ | S_m | \left| 
\fint_{S_m} \scjump{\uh} \right|^2 - 
\alpha C_1 \| \mh \|^2_{-1/2,\h} - \alpha C_2\sumi | \uh^i |_{1,\Fi}^2 -
\alpha C_3  \| \mh \|_{-1/2,h} \| \lh \|_{-1/2,h} \\
\geq
(1 - \alpha \const4)  \summ | S_m | \left| 
\fint_{S_m} \scjump{\uh} \right|^2 - \alpha \const2
\sumi | \uh^i |_{1,\Fi}^2
  - \alpha \const5 \| \lh \|^2_{-1/2,\h}.
\end{multline*}
whence, if $\alpha$ is sufficiently small
\begin{equation}\label{coercconst}
\A(\uh,\lh;0,\mh) \geq \ccoercb \summ  | S_m | \left| 
\fint_{S_m} \scjump{\uh} \right|^2 - \alpha \const2 \sumi | \uh^i |_{1,\Fi}^2 - \alpha \const5 \| \lh \|^2_{-1/2,\h}.
\end{equation}
By combining \eqref{coercnonconst} and \eqref{coercconst} we can write
\begin{multline*}
	\A(\uh,\lh; \uh,\lh + \tau \mh ) = \A(\uh,\lh;\uh,\lh) + \tau \A(\uh,\lh; 0,\mh)\\
	 \geq 
	(\ccoerca - \tau \alpha \const2) \sumi | \uh^i |_{1,\Fi}^2 + 
	\alpha (\ccoerca  - \tau  \const5) \| \lh \|_{-1/2,\delta}^2 + \tau
	\ccoercb \summ | S_m | \left| 
	\fint_{S_m} \scjump{\uh} \right|^2.
\end{multline*}
Choosing $\tau = \min\{ \ccoerca / (2\alpha C_2), \ccoerca/(2 C_5)\}$ we obtain
\begin{multline*}
	\A(\uh,\lh; \uh,\lh + \tau \mh )  \geq \frac \ccoerca 2\sumi | \uh^i |_{1,\Fi}^2 + \frac \ccoerca 2  \alpha \| \lh \|_{-1/2,\h}^2
+ \tau \ccoercb \summ | S_m | \left| 
\fint_{S_m} \scjump{\uh} \right|^2 \\ \gtrsim \normoneh{\uh}^2 + \alpha \| \lh \|_{-1/2,\h}^2.\end{multline*}
\end{proof}


\subsection{Error bound}

 Let $\uI$ denote the projection of $\u$ onto $V_h^i$ with respect to the scalar product inducing the $\normoneh{\cdot}$ norm, and let $\l_\pi \in \Lh$ and $\phI \in \Phh$ be defined as
\[
\int_\Sqi \l^i_\pi \mh^i = \int_\Sqi \l  \mh^i, \quad \forall \mh^i \in \Lh^i, \qquad \sumi  \h_i \int_{\Sqi} \phI \psh = \sumi   \h_i \int_{\Sqi} \ph \psh, \quad \forall \psh \in \Phh. 
\]
Since $\Lh^i = \Phh|_\Sqi$, this implies that
\[
\sumi  \l^i_\pi 
\psh =  \sumi  \l^i
\psh = 0,
\]
that is, $(\uI,\l_\pi)  \in \ker \B$.
Thanks to Lemma \ref{infsupkerB}, there exists a $\mh \in \ker \B$ such that
\[
\normoneh{\uh - \uI}+ \| \lh - \lI \|_{-1/2,\h} \lesssim \frac {\A(\uh - \uI,\lh - \lI;\uh - \uI,\mh)} { \normoneh{\uh - \uI} + \| \mh \|_{-1/2,\h}}
\]
Now we can write
\begin{multline*}
\A(\uh - \uI,\lh - \lI;\uh - \uI,\mh) = \A(\uh - \uI,\lh - \lI;\uh - \uI,\mh) + \underbrace{\B(\uh - \uI,\mh;\phh - \phI)}_{=0} \\ =
\underbrace{\A(\uh - \u,\lh - \l;\uh - \uI,\mh) + \B(\uh - \u,\mh;\phh - \ph)}_{=0} \\+ 
\A(\u - \uI,\l - \lI;\uh - \uI,\mh) + \B(h - \uI,\mh;\ph - \phI)\\
 \lesssim \Big(\normoneh{\u - \uI} + \| \l - \lI \|_{-1/2,\h}\Big)\Big( \normoneh{\uh - \uI} + \| \lh - \lI \|_{-1/2,\h} \Big) \\+ \| \mh \|_{-1/2,\h} \| \ph - \phI \|_{1/2,\h},
\end{multline*}
which yields
\[
\normoneh{\uh - \uI} + \| \lh - \lI \|_{-1/2,\h}  \lesssim \normoneh{\u - \uI} + \| \l - \lI \|_{-1/2,\h} + \| \ph - \phI \|_{1/2,\h}.
\]
The thesis follows thanks to our choice of $\uI$, $\lI$ and $\phI$, which minimizes  the right hand side of the above bound over all discrete functions.

\section{Numerical results}\label{sec:numtests}
The present section is devoted to the presentation of the performances of the above stabilizations. All tests are performed using linear Lagrangian Finite Elements on triangular meshes.

\begin{figure}
\centering
\includegraphics[width=0.75\textwidth]{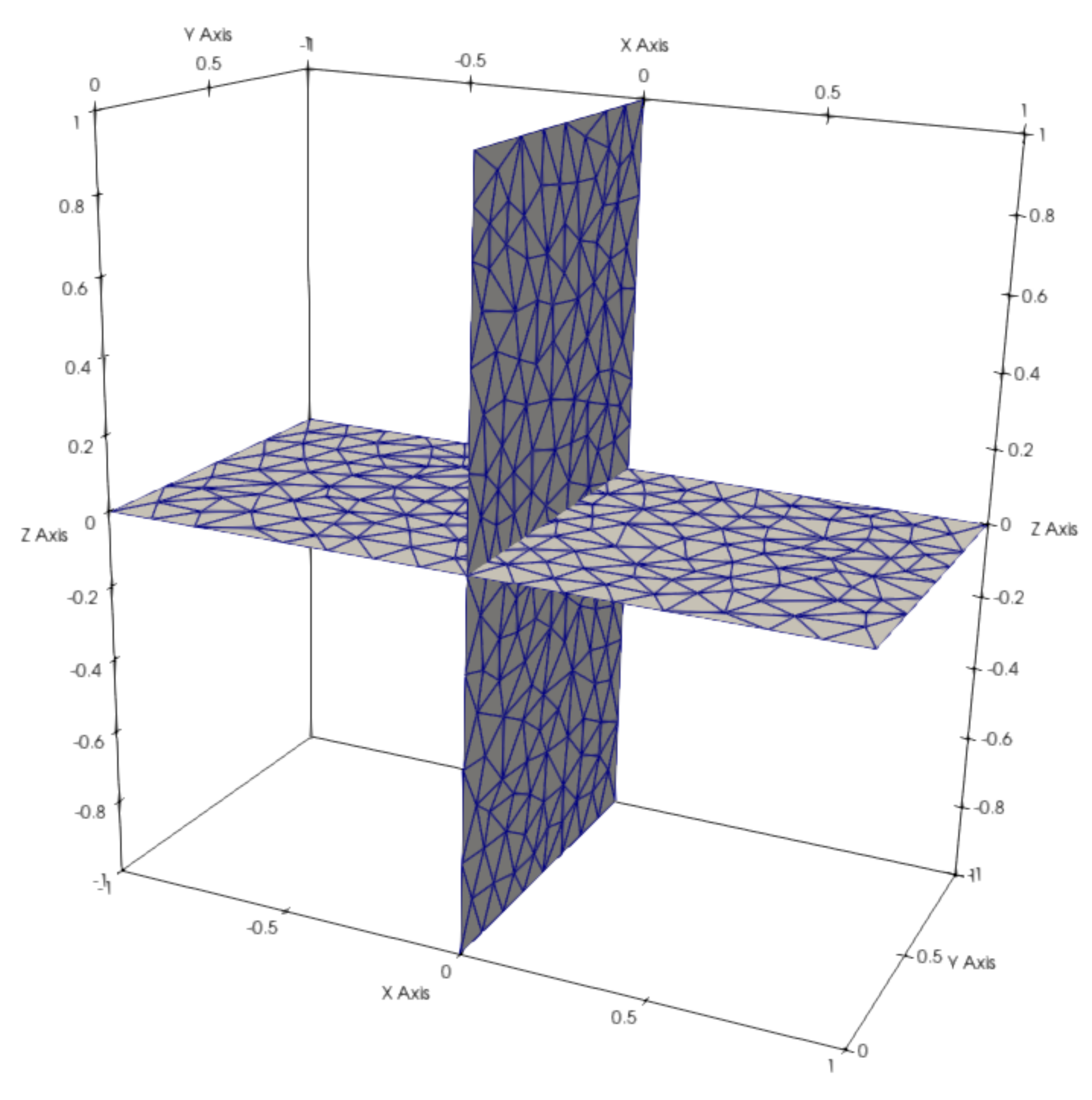}
\caption{\textit{test1}: Domain with a mesh ($\delta=0.1$) non conforming with fracture intersection.}
\label{test1Dom}
\end{figure}
\begin{figure}
\centering
\includegraphics[width=0.55\textwidth]{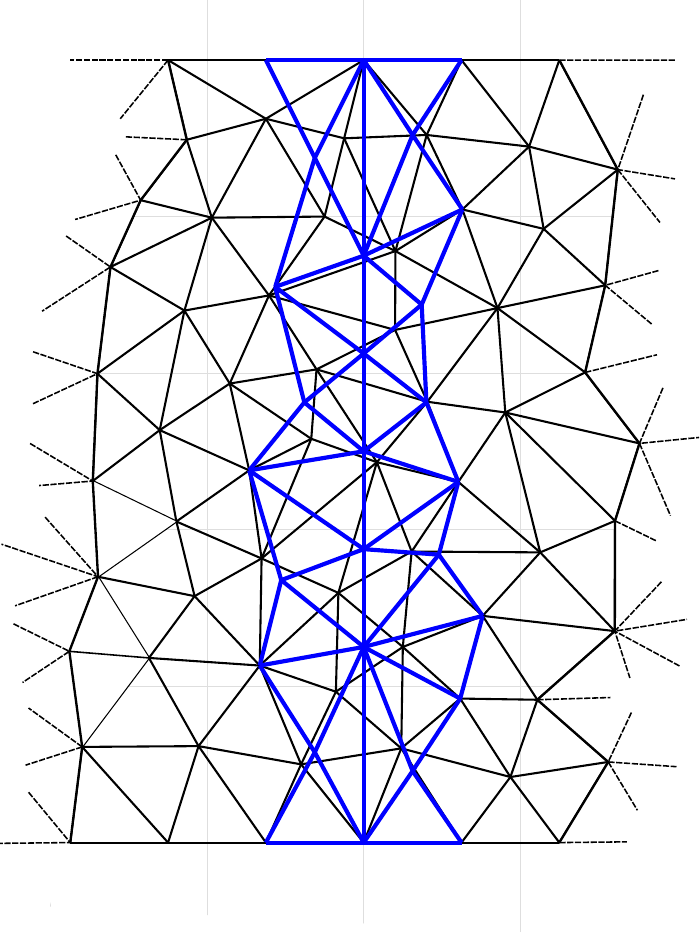}
\caption{\textit{test1}: Detail of mesh $\hat{\mathcal{T}}$ on fracture $F_1$, $\delta=0.1$.}
\label{test1StabMesh}
\end{figure}  

The first test, labelled \textit{test1} takes into account a network composed of two intersecting fractures, as shown in Figure~\ref{test1Dom}. We denote by $F_1$ the fracture lying on the plane $z=0$, by $F_2$ the fracture on the plane $x=0$ and by $\Sigma=\bar{F_1}\cap\bar{F_2}$ the intersection segment. Homogeneous Dirichlet boundary conditions are prescribed on the whole boundary and a forcing term is chosen such that the exact solution is:
\begin{displaymath}
\begin{cases}
h^1=y(y-1)|x|(x^2-1);\\
h^2=y(1-y)|z|(z^2-1).
\end{cases}
\end{displaymath}

With the above data we then solve Problem~\ref{pb:discrete-stabilized} and Problem~\ref{pb:3bBHstab} on five different meshes with mesh parameter $\delta=\{0.22,0.1, 0.071 ,0.032,0.014\}$, with $t=0$ and $\alpha=\{10,1,0.1,0.01,0.001\}$. Figure~\ref{test1StabMesh} shows mesh $\hat{\mathcal{T}}$ on $F_1$ used to construct the stabilizations, for a mesh parameter $\delta=0.1$. We also solve the same problem with a PDE-constrained optimization formulation \cite{BPSc}, on the same meshes. In all cases the resulting linear systems are solved using a direct method. 
Figures~\ref{test1L2err}-\ref{test1H1err} report the $L^2$ and $H^1$ norms of the error, respectively, of the obtained solutions with respect to the analytical one. In these figures, the black solid line refers to the reference solution whereas the other solid curves refers to the solutions of Problem~\ref{pb:discrete-stabilized} with different values of $\alpha$ and the dashed lines to the solutions of Problem~\ref{pb:3bBHstab} with the same values of $\alpha$. It can be observed that for $\omega \leq 1$ the error curves lie below those of the reference solution, and display a similar convergence trend. Sub-optimal convergence rates for linear FE are observed, due to the non-conformity of the mesh with the interface. The condition number of the system for the various considered values of $\alpha$ is reported in Figure~\ref{test1Cond}. We can see that almost all curves show a minimum of the conditioning for $\alpha=1$. In this case the system related to Problem~\ref{pb:3bBHstab} appears to have a better conditioning with respect to that of Problem~\ref{pb:discrete-stabilized}. The obtained solution of Problem~\ref{pb:discrete-stabilized} on fracture $F_1$ with the mesh having parameter $\delta=0.1$ is finally shown in Figure~\ref{test1SolF1}

\begin{figure}
\centering
\includegraphics[width=0.75\textwidth]{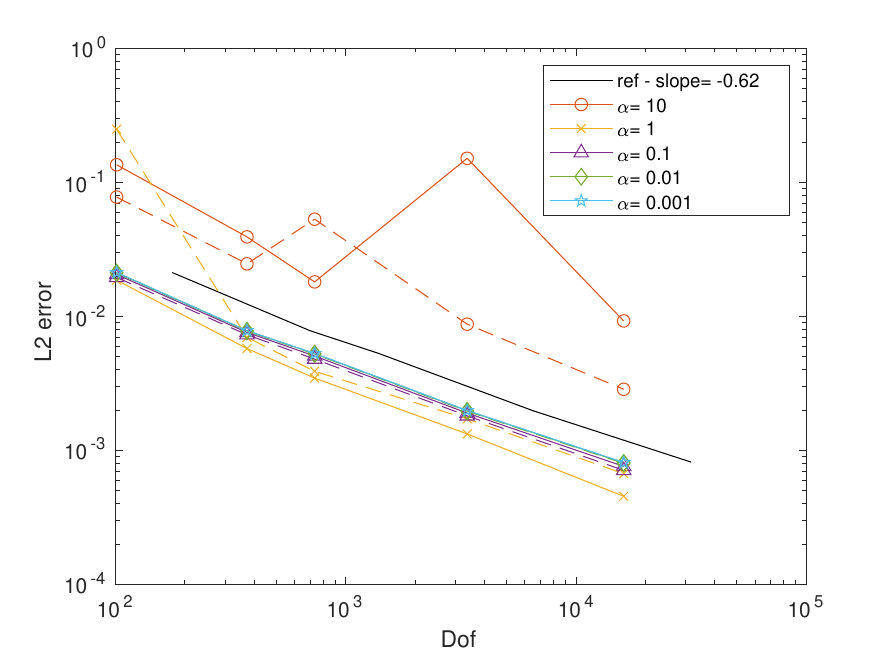}
\caption{\textit{test1}: $L^2$ error for different values of the stabilization parameter $\alpha$, for both the natural norm stabilization (solid line) and the mesh dependent stabilization (dashed line).}
\label{test1L2err}
\end{figure}
\begin{figure}
\centering
\includegraphics[width=0.75\textwidth]{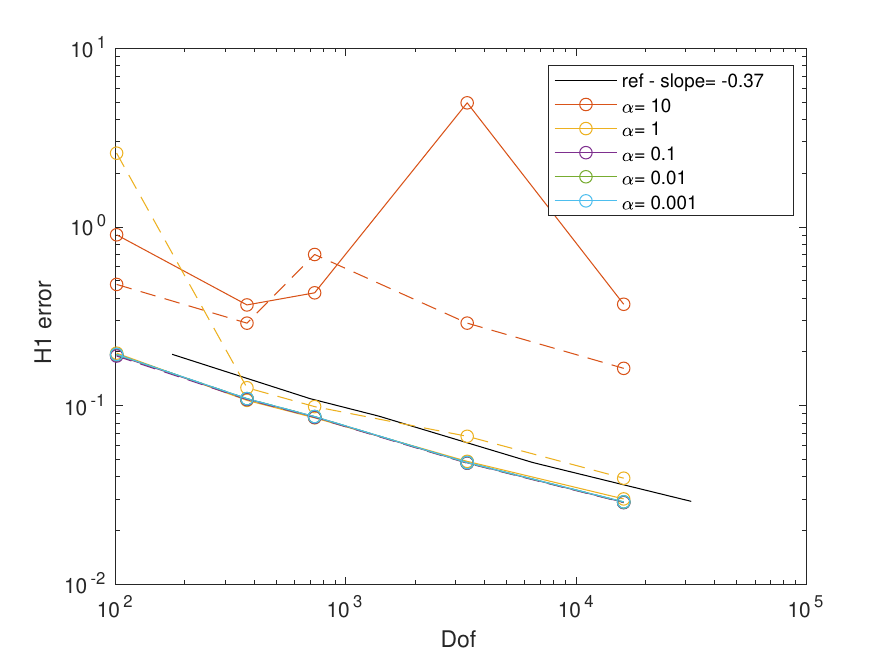}
\caption{\textit{test1}: $H^1$ error  for different values of the stabilization parameter $\alpha$, for both the natural norm stabilization (solid line) and the mesh dependent stabilization (dashed line).}
\label{test1H1err}
\end{figure}

\begin{figure}
\centering
\includegraphics[width=0.75\textwidth]{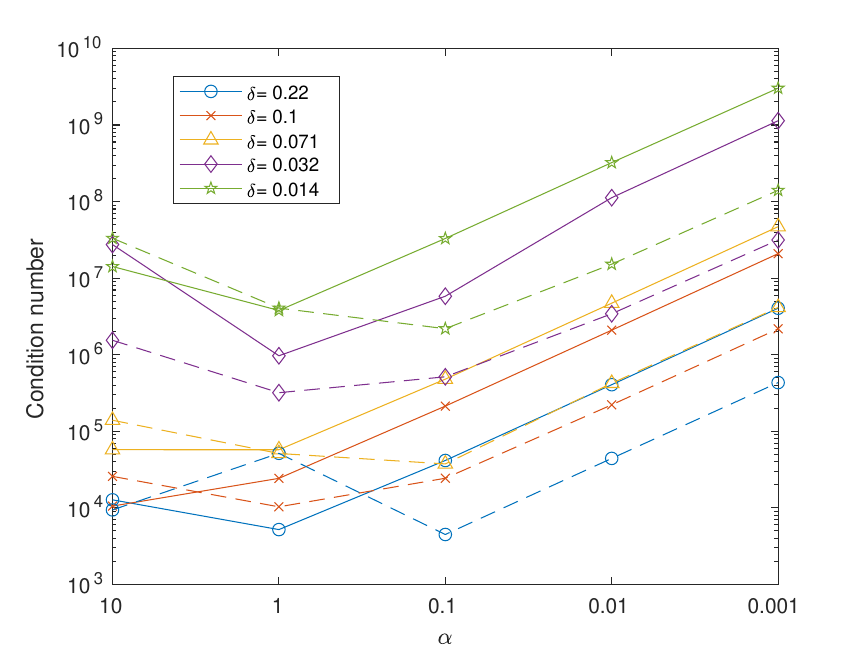}
\caption{\textit{test1}: Condition number  for different values of the stabilization parameter $\alpha$, for both the natural norm stabilization (solid line) and the mesh dependent stabilization (dashed line). }
\label{test1Cond}
\end{figure}
\begin{figure}
\centering
\includegraphics[width=0.75\textwidth]{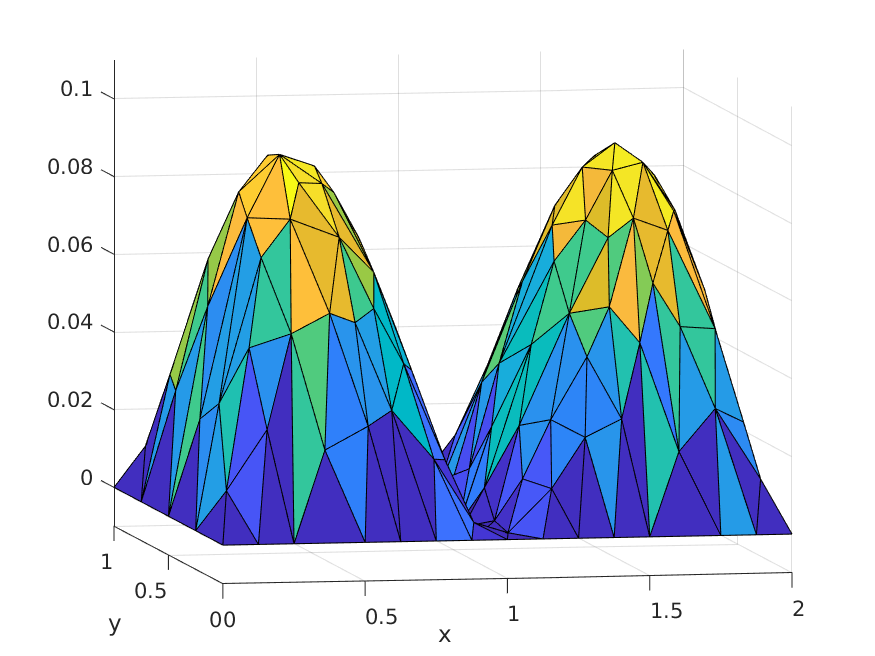}
\caption{\textit{test1}: Solution on $F_1$, $\delta=0.1$, $\omega=0.1$.}
\label{test1SolF1}
\end{figure}

The second test case takes into account a network of three intersecting fractures, with two slightly different configurations, as shown in Figure~\ref{test2Domain}. With reference to this figure, fracture $F_1$ lies on $z=0$, with edge length along $x$ equal to $2$ and edge length along $y$ equal to one. Fractures $F_2$ and $F_3$ are perpendicular to $F_1$ and are placed symmetrically with respect to the origin of a reference frame passing through the center of $F_1$. In the configuration of Figure~\ref{test2Domain}, left, the distance between $F_2$ and $F_3$ is $0.4$, and the stabilization meshes (shown in coloured thicker lines) are disjoint. We denote this case as \textit{test2A}. In the right configuration, instead, the distance between $F_2$ and $F_3$ is $0.05$ and the two stabilization meshes are partially overlapped, with the only exception of the finest considered mesh, thus violating the assumptions made in Section \ref{sec:StabTerm} for the construction of the stabilization term. This is termed \textit{test2B}. In both cases a unit/homogeneous Dirichlet boundary condition is placed on the top edge of fracture $F_2$/$F_3$, respectively, and homogeneous Neumann boundary conditions are prescribed on the remaining part of the boundary of the network. The forcing term is null, such that the analytic solution can be easily computed.
Figures~\ref{L2errA}-\ref{condA} refer to \textit{test2A}. Figures~\ref{L2errA}-\ref{H1errA} show convergence curves of the error with respect to the analytic solution in $L^2$ and $H^1$ norms against mesh refinement and for values of $\alpha=\{1,0.1,0.01,0.001,0.0001\}$. Figure~\ref{condA} shows the conditioning of the algebraic system for various meshes and for the same values of $\alpha$. As before, solid lines refer to Problem~\ref{pb:discrete-stabilized} and dashed lines to Problem~\ref{pb:3bBHstab}, whereas the black solid line corresponds to the reference solution obtained with the optimization based approach. We can observe a behaviour similar to the previous case. Higher values of $\alpha$ provide the lowest conditioning, and error convergence curves comparable with the reference are obtained. Notably, for $\omega=1$, the natural norm stabilization in Problem~\ref{pb:discrete-stabilized} provides lower errors compared to the mesh dependent stabilization in Problem~\ref{pb:3bBHstab}.
Figures~\ref{L2errB}-\ref{condB} refer instead to \textit{test2B}. Figures~\ref{L2errB}-\ref{H1errB} report convergence curves and Figure~\ref{condB} the conditioning of the discrete problem. If compared to the case with disjoint stabilization meshes, in \textit{test2B} higher values of both the $L^2$ and the $H^1$ errors are obtained, and also a worse conditioning is observed. Again the natural norm stabilization in Problem~\ref{pb:discrete-stabilized} with $\omega=1$ provides lower errors compared to the mesh dependent stabilization in Problem~\ref{pb:3bBHstab}, with a higher gap, especially on the coarser meshes.  This suggests that, when violating the simplifying assumption that the traces are well separated, the natural norm stabilization should be preferred to the mesh dependent one.

\begin{figure}
\centering
\includegraphics[width=0.75\textwidth]{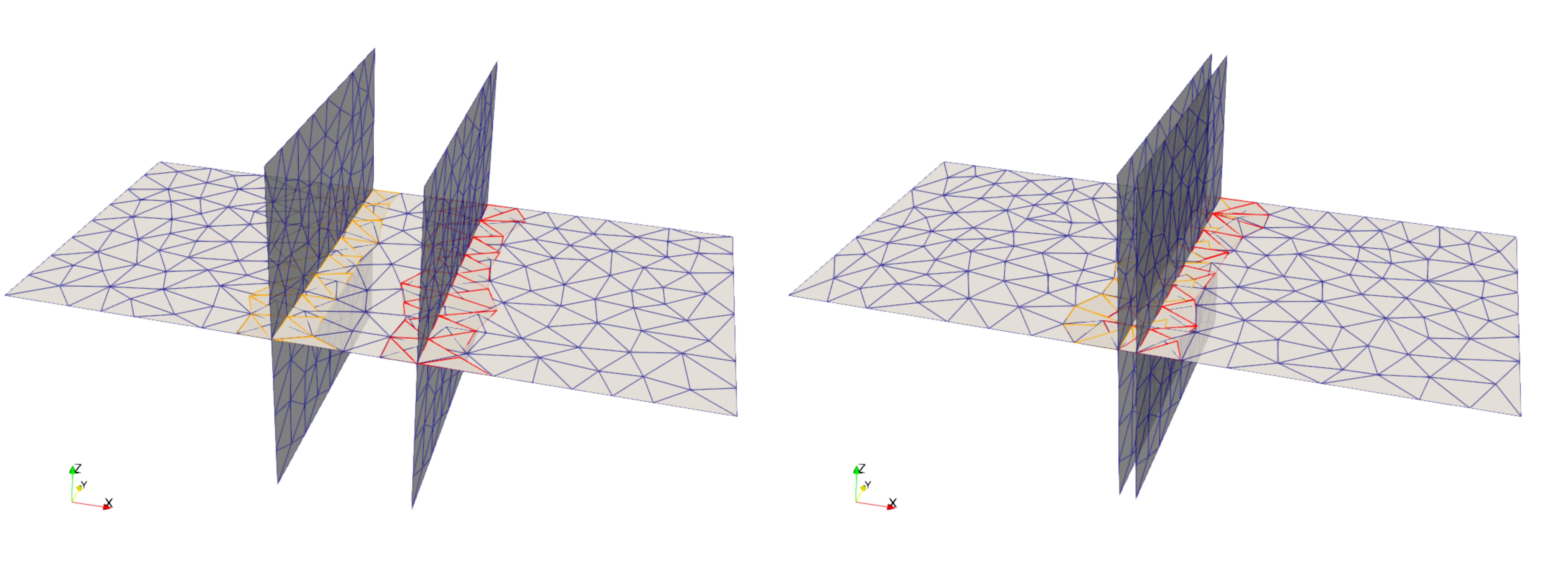}
\caption{\textit{test2}: domain and mesh for configurations A (left) and B (right)}
\label{test2Domain}
\end{figure}

\begin{figure}
\centering
\includegraphics[width=0.75\textwidth]{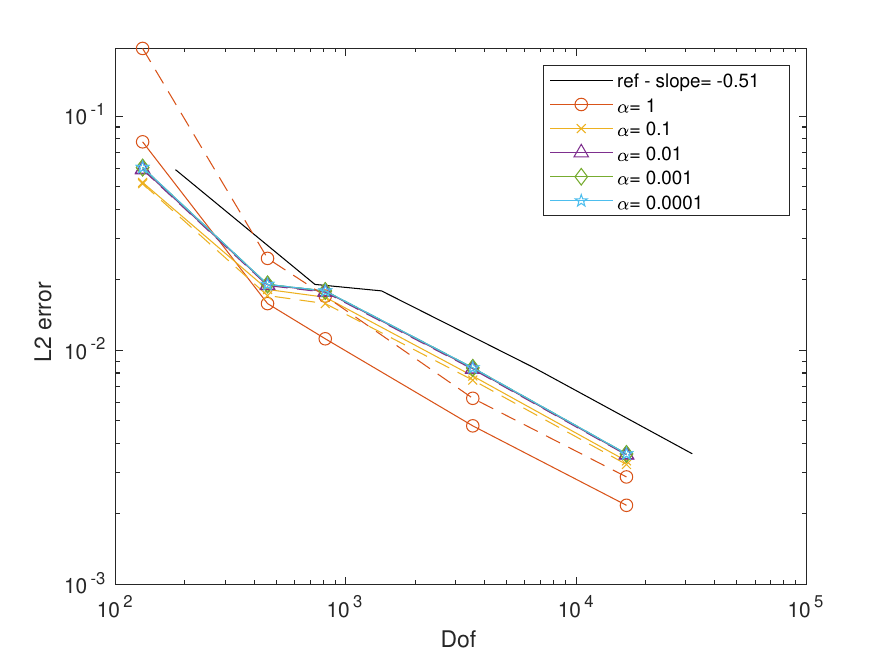}
\caption{\textit{test2A}: $L^2$ error  for different values of the stabilization parameter $\alpha$, for both the natural norm stabilization (solid line) and the mesh dependent stabilization (dashed line).}
\label{L2errA}
\end{figure}

\begin{figure}
\centering
\includegraphics[width=0.75\textwidth]{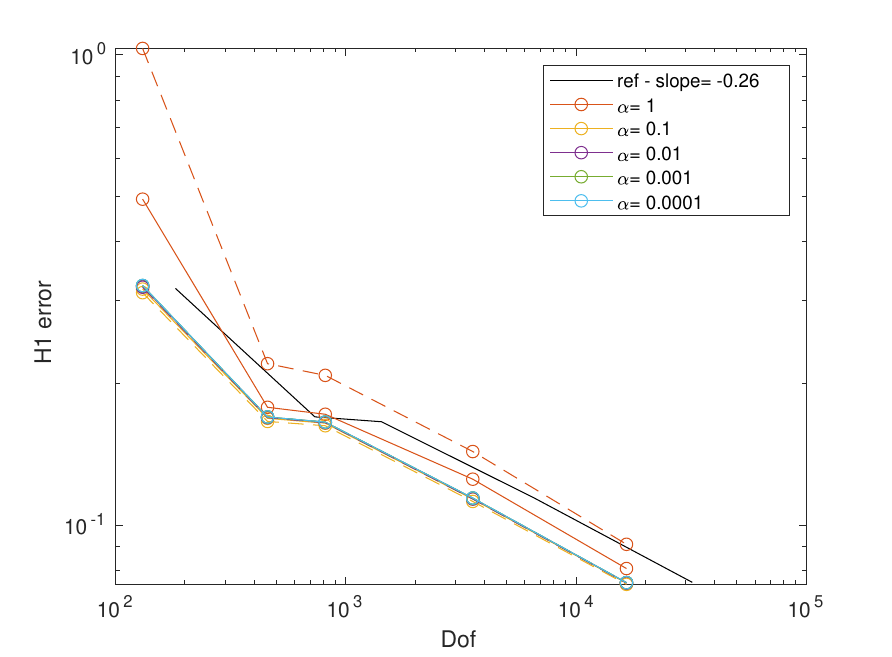}
\caption{\textit{test2A}: $H^1$ error for different values of the stabilization parameter $\alpha$, for both the natural norm stabilization (solid line) and the mesh dependent stabilization (dashed line).}
\label{H1errA}
\end{figure}

\begin{figure}
\centering
\includegraphics[width=0.75\textwidth]{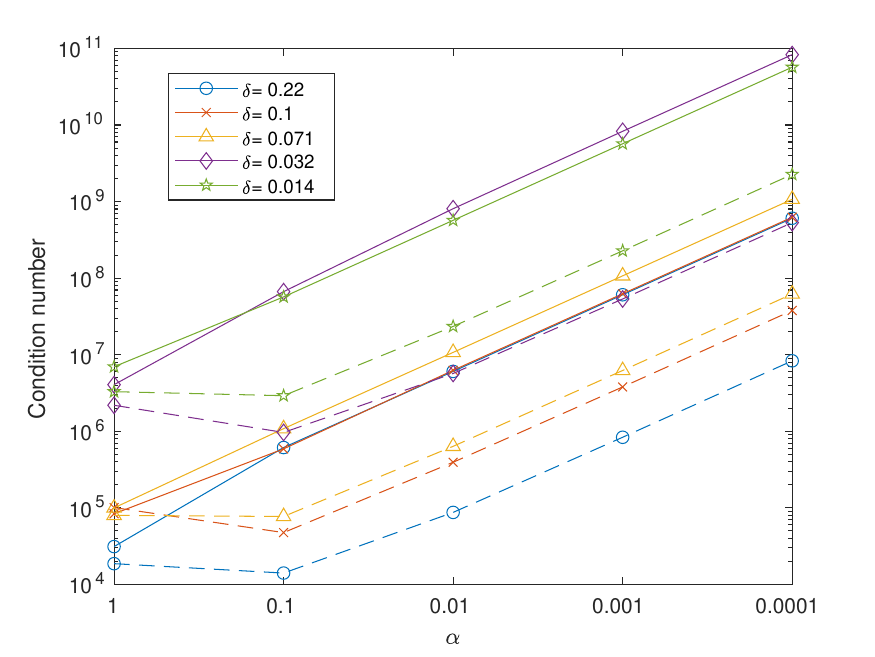}
\caption{\textit{test2A}: System conditioning for different meshes and various values of $\omega$.}
\label{condA}
\end{figure}

\begin{figure}
\centering
\includegraphics[width=0.75\textwidth]{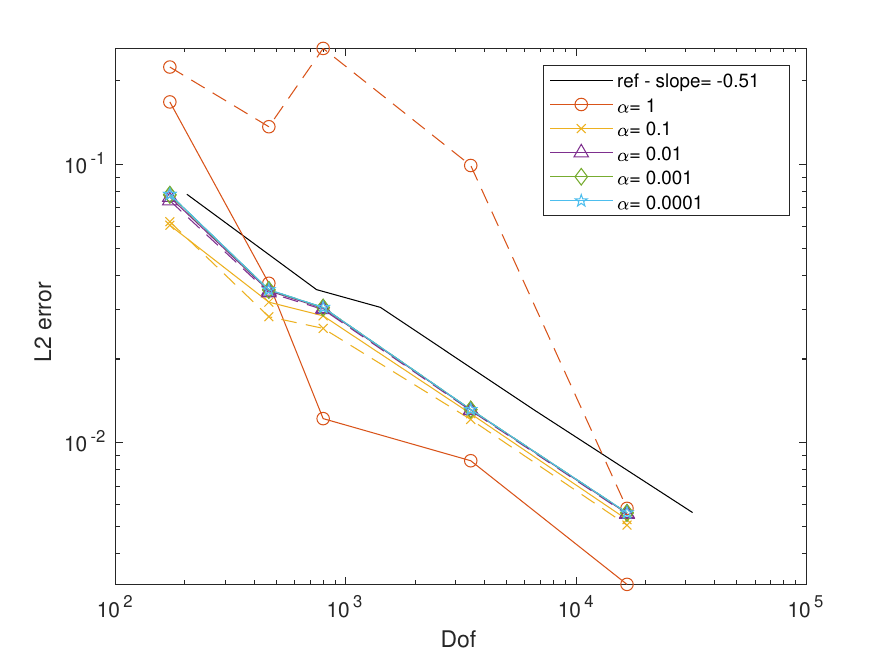}
\caption{\textit{test2B}: $L^2$ error for different values of the stabilization parameter $\alpha$, for both the natural norm stabilization (solid line) and the mesh dependent stabilization (dashed line).}
\label{L2errB}
\end{figure}

\begin{figure}
\centering
\includegraphics[width=0.75\textwidth]{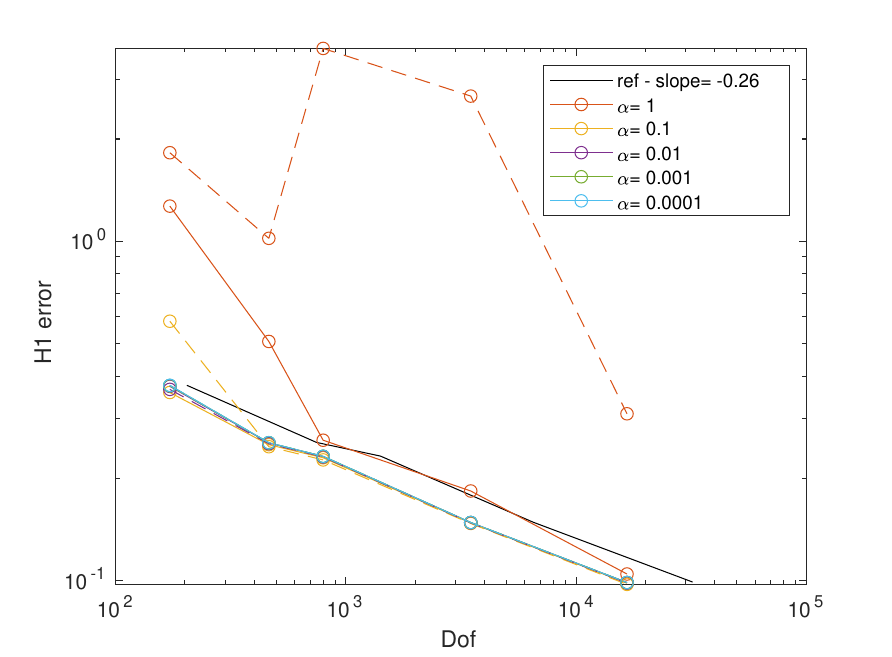}
\caption{\textit{test2B}: $H^1$ error for different values of the stabilization parameter $\alpha$, for both the natural norm stabilization (solid line) and the mesh dependent stabilization (dashed line).}
\label{H1errB}
\end{figure}

\begin{figure}
\centering
\includegraphics[width=0.75\textwidth]{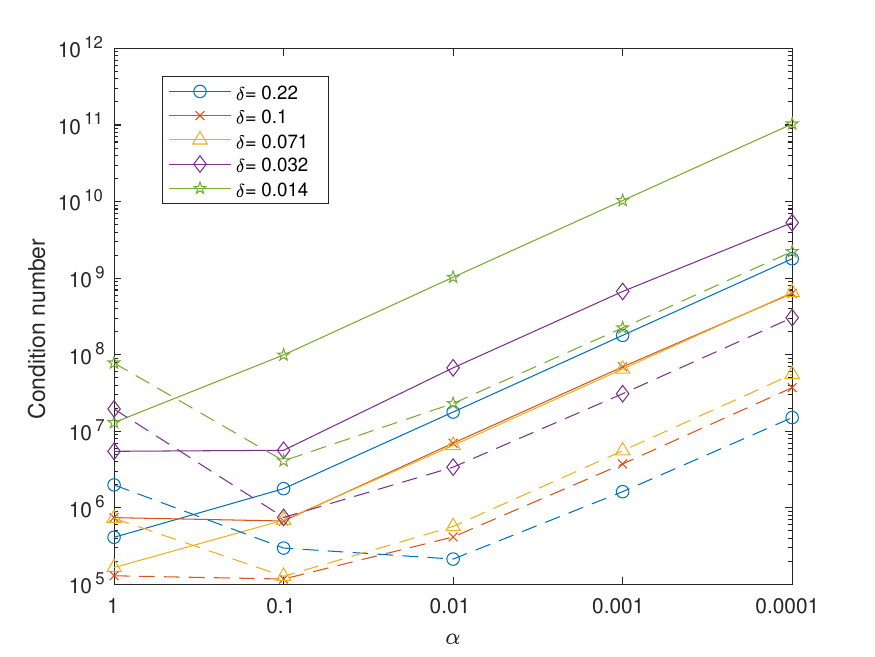}
\caption{\textit{test2B}: System conditioning for different meshes and various values of $\omega$.}
\label{condB}
\end{figure}

\section{Conclusions and future development}
We presented a hybridized domain decomposition method for the simulation of discrete fracture networks, for which we proposed two different stabilization strategies, respectively penalizing the residual in the natural norm and in a mesh dependent norm, the latter being a generalization of the Barbosa-Hughes stabililization to the case of an interface not aligned with the mesh. For both stabilization we proved stability and an optimal error estimate. The resulting numerical method has the potential to yield efficient discretization of DFN, that we plan to explore in the future. Indeed, in the domain decomposition framework, the formulation allows for the design of highly efficient preconditioners, which we hope to generalize to the DFN framework.
Several issues still remain to be addressed. In particular, the analysis of the method was carried out under the assumption that the different intersections are ``well separated'' from each other. Extending the theoretical results to more realistic case where two components of the interface can be very close to each other or even cross each other, possibly at a very small angle, will require either constructing a combined auxiliary space or to handle the sum of non linearly independent auxiliary spaces. 
Numerical test seem to suggest that in such cases the natural norm stabilization is to be preferred to the mesh dependent one.
We will address this non trivial issue in a future work.

\bibliographystyle{plain}
\bibliography{biblio}

\end{document}